%% file: main.tex
\documentclass[11pt]{article}
\usepackage[T1]{fontenc}

\usepackage{graphicx}
\usepackage{amsmath}
\usepackage{amsthm}
\usepackage{amsfonts}
\usepackage{amssymb}
\usepackage{thmtools}
\usepackage{mathtools}

\usepackage{tikz}
\usepackage{tikz-cd}

\usepackage{setspace}  
\usepackage{stfloats}
\usepackage{lmodern}
\usepackage{microtype}

\usepackage[shortlabels]{enumitem}
\setlist[itemize]{topsep=0ex,itemsep=0ex,parsep=0.4ex,leftmargin=1.2em}
\usepackage{caption}
\setlist[enumerate]{topsep=0ex,itemsep=0ex,parsep=0.4ex}
\usepackage[dvipsnames,svgnames,table]{xcolor}

\usepackage[unicode=true]{hyperref}
\hypersetup{
  colorlinks=true,
  breaklinks=true,
  linkcolor={blue},
  citecolor={blue},
  urlcolor={Navy},
  pdftitle={Verifying Hadwiger's Conjecture for Examples of Graphs with \texorpdfstring{$\alpha = 2$}{alpha = 2}},
  pdfauthor={Jofre Costa, Eric Liu, David R. Wood, Jung Hon Yip}
}

\usepackage[noabbrev,capitalise]{cleveref}
\crefname{lem}{Lemma}{Lemmas}
\crefname{thm}{Theorem}{Theorems}
\crefname{cor}{Corollary}{Corollaries}
\crefname{prop}{Proposition}{Propositions}
\crefname{conj}{Conjecture}{Conjectures}
\crefname{obs}{Observation}{Observations}
\crefname{quest}{Question}{Questions}
\crefname{open}{Open Problem}{Open Problems}
\crefname{claim}{Claim}{Claims}

\numberwithin{equation}{section}
\theoremstyle{plain}
\newtheorem{thm}[equation]{Theorem}
\newtheorem{lem}[equation]{Lemma}
\newtheorem{cor}[equation]{Corollary}

\newtheorem{obs}[equation]{Observation}

\theoremstyle{definition}
\newtheorem{conj}[thm]{Conjecture}
\newtheorem{quest}[thm]{Question}
\newcounter{MyEqNum}

\usepackage[margin=30mm]{geometry}
\usepackage{float}

\DeclarePairedDelimiter{\floor}{\lfloor}{\rfloor}
\DeclarePairedDelimiter{\ceil}{\lceil}{\rceil}

\renewcommand{\geq}{\geqslant}
\renewcommand{\leq}{\leqslant}
\DeclareMathOperator{\diam}{diam}
\DeclareMathOperator{\dist}{dist}
\DeclareMathOperator{\had}{had}

\newcommand{\PP}{\mathbb{P}}

\newcommand{\Q}{\mathbb{Q}}

\newcommand{\N}{\mathbb{N}}
\newcommand{\F}{\mathbb{F}}

\DeclareMathOperator{\HC}{HC}
\DeclareMathOperator{\srg}{srg}
\DeclareMathOperator{\SHC}{SHC}

\DeclareMathOperator{\CDM}{CDM}

\DeclareMathOperator{\icr}{icr}
\DeclareMathOperator{\Cay}{Cay}
\DeclareMathOperator{\cm}{cm}
\newcommand{\cmark}{\ding{51}}
\newcommand{\qmark}{\textbf{?}}

\renewcommand{\thefootnote}{\fnsymbol{footnote}}
\usepackage[bottom]{footmisc}

\usepackage{tabularx}
\usepackage{multirow}
\usepackage{pifont}

\usepackage[longnamesfirst,numbers,sort&compress]{natbib}
\makeatletter
\def\NAT@spacechar{~}
\makeatother
\newcommand{\defn}[1]{\textcolor{green!50!black}{\textit{#1}}}
\newcommand{\mathdefn}[1]{\textcolor{green!50!black}{#1}}
\newcommand{\conjdefn}[1]{\textcolor{red!50!black}{#1}}

\newcounter{propnum}
\renewcommand{\thepropnum}{\arabic{propnum}} 
\makeatletter
\newcommand{\propertylabel}[2]{%
  (\stepcounter{propnum}\phantomsection
   \protected@edef\@currentlabel{\thepropnum}
   \label{#1}%
   \thepropnum): #2}
\makeatother
\newcommand{\propref}[1]{\textup{(\ref{#1})}}

\begin{document}


\author
{
    Jofre Costa\footnotemark[1] \qquad
    Eric Luu\footnotemark[2] \qquad
    David R. Wood\footnotemark[2] \qquad
    Jung Hon Yip\footnotemark[2]
}

\footnotetext[1]{Department of Mathematics, University of Bonn, Germany (\texttt{jofrecosta01@gmail.com}). Research completed at Monash University. }

\footnotetext[2]{School of Mathematics, Monash University, Melbourne, Australia, \texttt{ericluu42@gmail.com}, \texttt{David.Wood@monash.edu}, \texttt{Junghon.Yip@monash.edu}. Research of Wood is supported by the Australian Research Council and by NSERC. Yip is supported by a Monash Graduate Scholarship.}

\title{\bf\boldmath
    Verifying Hadwiger's Conjecture for \\
    Examples of Graphs with $\alpha(G)=2$}
\maketitle

\begin{abstract}
    Hadwiger's Conjecture states that every graph with chromatic number $k$ contains a complete graph on $k$ vertices as a minor. This conjecture is a tremendous strengthening of the Four-Colour Theorem and is regarded as one of the most important open problems in graph theory. The case of Hadwiger's Conjecture for graphs with $\alpha(G) = 2$ has garnered much attention. Seymour writes: ``My own belief is, if Hadwiger's Conjecture is true for graphs with stability number two then it is probably true in general, so it would be very nice to decide this case.''

    This paper presents several tools useful for proving that a graph $G$ with $\alpha(G) = 2$ satisfies Hadwiger's Conjecture. In doing so, we survey and generalise several classical results on the $\alpha(G) = 2$ case of Hadwiger's Conjecture. Further, we apply these tools to prove variants of Hadwiger's Conjecture for several noteworthy classes of graphs with $\alpha(G) = 2$. In particular, we prove Hadwiger's Conjecture for inflations of the complements of the following graphs: graphs with girth at least $5$, triangle-free Kneser graphs, and the Clebsch, Mesner, and Gewirtz graphs.
    This paper also highlights classes of graphs with $\alpha(G) = 2$ where it is unknown if Hadwiger's Conjecture holds.
\end{abstract}

\renewcommand{\thefootnote}
{\arabic{footnote}}


\section{\boldmath Introduction}
\label{Introduction}
A graph $H$ is a \defn{minor} of a graph $G$ if $H$ can be obtained from $G$ by deleting vertices and edges, and by contracting edges. The maximum integer $n$ such that the complete graph $K_n$ is a minor of $G$ is called the \defn{Hadwiger number} of $G$, denoted \defn{$\had(G)$}. A \defn{proper colouring} of a graph $G$ assigns a colour to each vertex so that adjacent vertices receive distinct colours. The \defn{chromatic number} of $G$, denoted \defn{$\chi(G)$}, is the minimum integer $k$ such that $G$ has a proper colouring with $k$ colours. For any undefined graph theory notation, see \citet{Diestel05}\footnote{Unless stated otherwise, all graphs are simple and finite.}.

Hadwiger's Conjecture is widely considered to be one of the deepest unsolved problems in graph theory~\citep{Hadwiger43}. See \citep{Kawa15,Toft96,SeymourHC,CV20} for surveys. It asserts that:
\begin{conj}
    \label{conj:hc}
    For every graph $G$, $\had(G) \geq \chi(G)$. Equivalently, if $G$ is $K_t$ minor-free, then $G$ is $(t-1)$-colourable.
\end{conj}
An \defn{independent set} in a graph $G$ is a set of pairwise non-adjacent vertices in $G$. The \defn{independence number} of $G$, denoted \defn{$\alpha(G)$}, is the cardinality of the largest independent set in $G$. Since $\chi(G) \alpha(G) \geq |V(G)|$ for every graph $G$, if \cref{conj:hc} were true, then $\had(G) \geq |V(G)|/\alpha(G)$. However, this weakening is open:
\begin{conj}
    \label{conj:indep}
    For every graph $G$, $\had(G) \geq |V(G)|/\alpha(G)$.
\end{conj}
Several results in the direction of \cref{conj:indep} are known \citep{DM82,Fox10,BK11,bohme2011minors, KawaHadwigerOdd07, FradkinClique12}. A classical result of \citet{DM82} states that:
\begin{thm}[\citet{DM82}]
    \label{thm:dm}
    For every graph $G$, $\had(G) \geq  |V(G)|/ (2\alpha(G) -1)$.
\end{thm}
There have been small improvements to the constant factor in \cref{thm:dm}, notably by \citet{Fox10}, and then by \citet{BK11}, who proved that:
\begin{thm}[\citet{BK11}]
    \label{thm:complete_minors_indep_no}
    For any graph $G$, $\had(G) \geq |V(G)|/((2-c)\alpha(G))$, where $c = 1/19.2$.
\end{thm}
If $\alpha(G) \geq 19$, then \cref{thm:complete_minors_indep_no} offers an improvement over \cref{thm:dm}. However, when $\alpha(G) < 19$, it appears difficult to improve the constant factor in \cref{thm:dm}. In particular, when $\alpha(G) = 2$, \cref{thm:dm} implies that $\had(G) \geq |V(G)|/3$, and \cref{thm:complete_minors_indep_no} does not offer a better bound. However, in this case \cref{conj:indep} asserts that $\had(G) \geq |V(G)|/2$. This gap is large, and thus the restriction of Hadwiger's Conjecture to graphs with independence number $2$ is interesting. \citet{SeymourHC} writes ``My own belief is, if Hadwiger's Conjecture is true for graphs with stability number two then it is probably true in general, so it would be very nice to decide this case''. For convenience, we denote this special case as \nameref{conj:hc_chi}.

\begin{conj}[\conjdefn{$\HC_{\alpha = 2}$}]
    For every graph $G$ with $\alpha(G) = 2$, $\had(G) \geq \chi(G)$. \label{conj:hc_chi}
\end{conj}
\nameref{conj:hc_chi} is a key open case of Hadwiger's Conjecture~\citep{PST03,CS12,Blasiak07,CO08,Bosse19,NS26a}. Every graph $G$ with $\alpha(G) = 2$ has $\chi(G) \geq |V(G)|/2$, but the weakening of \nameref{conj:hc_chi} remains open:
\begin{conj}[\conjdefn{$\HC_{n/2}$}]
    For every graph $G$ with $\alpha(G) = 2$, $\had(G) \geq |V(G)|/2$. \label{conj:hc_half}
\end{conj}
\citet{PST03} proved that a minimum counterexample (in terms of $|V(G)|$) to \nameref{conj:hc_chi} has several properties (see \cref{tab:minimal_properties}).  In particular, they proved that if $G$ is a minimum counterexample to \nameref{conj:hc_chi}, then $\diam(\overline{G}) = 2$, and $G$ is a counterexample to \nameref{conj:hc_half}. Therefore, \nameref{conj:hc_chi} and \nameref{conj:hc_half} are equivalent. (Several authors (\citep[Theorem 1.2]{Bosse19}, \citep[Theorem 1.3]{Zhou23}) have misinterpreted this statement as claiming that every graph that satisfies \nameref{conj:hc_half} also satisfies \nameref{conj:hc_chi}. This is not implied.)

Even the smallest improvement to \cref{thm:dm} in the $\alpha(G) = 2$ case would be interesting. In particular, \citet{SeymourHC} conjectured that the constant factor of $1/3$ in \cref{thm:dm} can be improved:
\begin{conj}[\conjdefn{$\HC_{\varepsilon}$}]
    \label{conj:hc_epsilon}
    There is a constant $\varepsilon > 0$ such that every graph $G$ with $\alpha(G) = 2$ satisfies $\had(G) \geq (\frac{1}{3} + \varepsilon) |V(G)|$.
\end{conj}
\nameref{conj:hc_epsilon} turns out to be equivalent to a conjecture about connected matchings, which we define below.

\begin{sloppypar}
    For disjoint subsets $A, B \subseteq V(G)$, $A$ and $B$ are \defn{adjacent} if some vertex in $A$ is adjacent to some vertex in $B$, otherwise $A$ and $B$ are \defn{non-adjacent}. Let $M \subseteq E(G)$. Let
    $\mathdefn{V(M)} := \{ v \in V(G) : \text{$v$ is an endpoint of an edge $e \in M$} \}$. For an edge $e \in E(G)$, let $\mathdefn{V(e)}:= V(\{e\})$.
    A matching $M \subseteq E(G)$ is \defn{connected} if for every two edges $e, f \in M$, $V(e)$ and $V(f)$ are adjacent. A
    matching $M \subseteq E(G)$ is \defn{dominating} if each vertex of $V(G) \setminus V(M)$ is adjacent to each edge $e \in M$. An edge $xy \in E(G)$ is \defn{dominating} if $N_G(x) \cup N_G(y) = V(G)$. Equivalently, an edge $xy$ is dominating if $\{xy\}$ is a connected dominating matching of size $1$. \cref{conj:linear_cm} is due to \citet{furedi2005connected} (although they note that other authors may have independently made the same conjecture).
\end{sloppypar}

\begin{conj}[\conjdefn{Linear-CM}]
    \label{conj:linear_cm}
    There is a constant $c > 0$ such that every graph $G$ with $\alpha(G) = 2$ has a connected matching of size at least $c|V(G)|$.
\end{conj}
Thomass\'e first noted that \nameref{conj:linear_cm} and \nameref{conj:hc_epsilon} are equivalent (see \citep{furedi2005connected}). A proof is given by \citet{KPT05}. Therefore, improving the constant $1/3$ in \cref{thm:dm} is as hard as finding a connected matching of linear size. \citet{furedi2005connected} conjectured a precise version of \nameref{conj:linear_cm}.
\begin{conj}[\conjdefn{$\text{4-CM}$}]
    \label{conj:4-CM}
    For every integer $t \geq 1$, every graph $G$ with $\alpha(G) = 2$ with at least $4t - 1$ vertices has a connected matching of size at least $t$.
\end{conj}
The bound $4t - 1$ is best possible, since the disjoint union of two copies of the complete graph $K_{2t - 1}$ does not have a connected matching of size $t$. Somewhat surprisingly,
\citet{cambie2021hadwiger} proved that \nameref{conj:hc_chi} implies \nameref{conj:4-CM}.

Let $G$ be a graph. A \defn{model} of a graph $H$ in a graph $G$ is a function $\eta: V(H) \to 2^{V(G)}$ assigning the vertices of $H$ vertex-disjoint connected subgraphs of $G$, such that if $uv \in E(H)$, then there is an edge of $G$ between $\eta(u)$ and $\eta(v)$. The sets $\{\eta(u): u \in V(H)\}$ are called the \defn{branch sets} of the model. There exists a model of a graph $H$ in $G$ if and only if $H$ is a minor of $G$. Observe that a connected matching is a complete graph model in which each branch set has size exactly $2$. \citet{SeymourHC} conjectured the following strengthening of \nameref{conj:hc_chi}.
\begin{conj}[\conjdefn{$\SHC_{\alpha = 2}$}]
    \label{conj:shc_chi}
    Every graph $G$ with $\alpha(G) = 2$ has a $K_{\chi(G)}$-model in which each branch set has size at most $2$.
\end{conj}
This would imply:
\begin{conj}[\conjdefn{$\SHC_{n/2}$}]
    \label{conj:shc_half}
    Every graph $G$ with $\alpha(G) = 2$ has a $K_{\ceil{|V(G)|/2}}$-model in which each branch set has size at most $2$.
\end{conj}
The conjecture \nameref{conj:shc_chi} is known as Seymour's strengthening of Hadwiger's Conjecture in \citep{Blasiak07}, or the matching version of Hadwiger's Conjecture in \citep{CS12}. Our first contribution is to establish the equivalence of \nameref{conj:shc_chi} and \nameref{conj:shc_half}, following the proof techniques in \citet{PST03}. We conjecture the following strengthening of \nameref{conj:shc_chi} and \nameref{conj:shc_half}, which to our knowledge remains open:
\begin{conj}[\conjdefn{$\CDM$}]
    \label{conj:cdm}
    Every connected graph $G$ with $\alpha(G) = 2$ has a non-empty connected dominating matching.
\end{conj}
\noindent
\begin{minipage}{0.38\textwidth}
    We have computationally verified \nameref{conj:cdm} for all graphs up to $11$ vertices.
    In \cref{thm:equiv}, we prove that \nameref{conj:cdm} implies \nameref{conj:shc_chi} (and hence \nameref{conj:shc_half}, since \nameref{conj:shc_chi} and \nameref{conj:shc_half} are equivalent).
    We do not know if \nameref{conj:cdm} is equivalent to \nameref{conj:shc_chi}.
    Note that one cannot remove the assumption that $G$ is connected in \nameref{conj:cdm}, since the disjoint union of two non-empty complete graphs does not have a non-empty connected dominating matching.
    \cref{fig:equiv} describes the relations between the various conjectures.
\end{minipage}%
\hfill
\begin{minipage}{0.57\textwidth}
    \centering
    \input{implication_diagram.tex}
    
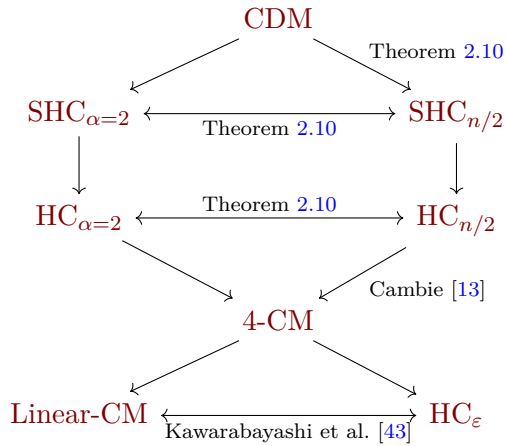
\captionof{figure}{Relations between the conjectures. An arrow from one conjecture to another indicates that the former implies the latter.}
    \label{fig:equiv}
\end{minipage}

\subsection{Inflations and Blow-Ups}
\label{ss:inflations}
Let $H$ be a graph and assign to each vertex $x \in V(H)$ an integer $c_x \geq 0$. Construct a new graph $G$ from $H$ by replacing each vertex $x \in V(H)$ with a clique $C_x$ of order $c_x$. Each vertex of $C_x$ is adjacent to each vertex of $C_y$ if and only if $xy \in E(H)$, otherwise $C_x$ and $C_y$ are non-adjacent. We call $G$ an \defn{inflation} of $H$. We define the \defn{projection} map $\mathdefn{p}: V(G) \to V(H)$ by $p(v) = x$ if $v \in C_x$. Note that $G$ could be an induced subgraph of $H$ by taking some $c_x = 0$. An inflation of $H$ is \defn{$k$-uniform} if $c_x = k$ for each $x \in V(H)$, and \defn{proper} if $c_x \geq 1$ for each $x \in V(H)$. Note that $p$ is surjective if and only if $G$ is a proper inflation of $H$.
Two vertices $u, v \in V(G)$ are \defn{adjacent-twins} if $uv \in E(G)$ and $N_G(u) \setminus \{v\} = N_G(v) \setminus \{u\}$. $G$ is \defn{adjacent-twin-free} if there are no adjacent-twins.
Observe that for every graph $G$, there is a unique smallest induced subgraph $H$ of $G$ such that $G$ is a proper inflation of $H$.

\begin{obs}
    \label{obs:preserve}
    If $G$ is a proper inflation of $H$, then $\alpha(G) = \alpha(H)$ and $\diam(\overline{G}) = \diam(\overline{H})$.
\end{obs}

By \cref{obs:preserve}, if a graph $H$ has $\alpha(H) = 2$, then any inflation $G$ of $H$ also has $\alpha(G) \leq 2$. A natural question arises:

\begin{conj}
    \label{conj:inflation}
    Let $G$ be a graph with $\alpha(G) = 2$, and suppose $G$ satisfies $\mathfrak{C}$, where
    $
        \mathfrak{C} \in \{\text{\nameref{conj:hc_chi}, \nameref{conj:hc_half}, \nameref{conj:hc_epsilon}, \nameref{conj:linear_cm}, \nameref{conj:4-CM}, \nameref{conj:shc_chi}, \nameref{conj:shc_half}, \nameref{conj:cdm}}\}.
    $
    Then any inflation of $G$ satisfies $\mathfrak{C}$.
\end{conj}
\cref{conj:inflation} remains open because the behaviour of minors under inflations is not well understood. See \citep{Pedersen12,CP16,Pedersen11} on Hadwiger's Conjecture for inflations of special classes of graphs.
\citet{Catlin79} showed that uniform inflations of the $5$-cycle are counterexamples to Hajós' Conjecture, which asserts that a graph with chromatic number $k$ contains $K_k$ as a topological minor. While inflations of the $5$-cycle have a non-empty connected dominating matching \citep{PST03}, it is possible an inflation of another graph could yield a counterexample to \nameref{conj:cdm}.

It is often easier to study inflations through their complements. Let $G$ and $H$ be graphs, and let $\overline{G}$ and $\overline{H}$ be their complements. $\overline{G}$ is a \defn{blow-up} of $\overline{H}$ if $G$ is an inflation of $H$. More precisely, if $G$ is an inflation of $H$ with $x \in V(H)$ replaced with a clique $C_x$ of order $c_x \geq 0$, then $\overline{G}$ is obtained from $\overline{H}$ by replacing each vertex $x \in V(\overline{H})$ with an independent set $A_x$ of order $c_x$ (so that each vertex of $A_x$ is adjacent to each vertex of $A_y$ if and only if $xy \in E(\overline{H})$). We use natural analogues of the terms \defn{proper}, \defn{$k$-uniform} and the \defn{projection} map for blow-ups. Two vertices $u, v \in V(G)$ are \defn{twins} if $uv \notin E(G)$ and $N_G(u) = N_G(v)$. $G$ is \defn{twin-free} if no two vertices of $G$ are twins.
We frequently use the following:
\begin{obs}
    $G$ is adjacent-twin-free if and only if $\overline{G}$ is twin-free.
\end{obs}

\subsection{Goals and Outline}
\label{ss:goals_outline}
The primary aim of this paper is to identify several noteworthy classes of graphs with independence number $2$ and to demonstrate that Hadwiger's Conjecture holds for their inflations; this is done in \cref{section:results_special} and summarised below. These graphs are chosen because they have many properties of a hypothetical counterexample to one of the conjectures above. The proofs rely on general-purpose techniques first surveyed in \cref{s:lit_review}, where we review and generalise several classical results on the $\alpha(G) = 2$ case of Hadwiger's Conjecture. Lastly, in \cref{s:op}, we present open problems and specific classes of graphs for which Hadwiger's Conjecture is open, and therefore might be possible counterexamples.

In each of the following graph classes, we state the strongest variant of Hadwiger's Conjecture that we are able to prove. See the respective sections for a more detailed discussion. Unproven cases are listed in \cref{ss:potential_counteg}.
\begin{enumerate}[label={},leftmargin=0em]
    \item Complements of Graphs with Girth at least $5$ (\cref{ss:girth5}):
          \begin{itemize}
              \item Any inflation of the complement of a graph with girth at least $5$ satisfies \nameref{conj:cdm} (\cref{thm:girth_5}).
              \item Let $G$ be the complement of the $5$-cycle, or the complement of the Petersen graph, or the complement of the Hoffman-Singleton graph, or the complement of a hypothetical $57$-regular graph with diameter $2$ and girth $5$. Then any inflation of $G$ satisfies \nameref{conj:cdm} (\cref{thm:girth_5_eg}).
          \end{itemize}
    \item Complements of Triangle-Free Generalised Kneser Graphs (\cref{ss:kneser_graphs}): For integers $n, k \geq 1$, let \defn{$\binom{[n]}{k}$} denote all $k$-subsets of $\{1, \dots, n\}$. For $t \geq 1$, let $K(n,k, \geq t)$ be the graph with vertex set $V(G) = \binom{[n]}{k}$, where two vertices are adjacent if and only if their corresponding subsets intersect in at least $t$ elements. Define $K(n,k, \leq t):= \overline{K(n,k, \geq t+1)}$ and $K(n,k,\leq 0):= K(n,k)$, known as the \defn{Kneser graph}. The graph $K(n, k, \leq t)$ is known as a \defn{generalised Kneser graph} in the literature.
          \begin{itemize}
              \item For integers $n,k \geq 1$ such that $2k \leq n \leq 3k - 1$, any inflation of $\overline{K(n,k)}$ with an even number of vertices satisfies \nameref{conj:shc_half} (\cref{lem:kneser_graphs}).
              \item For integers $n,k \geq 1$ such that $2k \leq n \leq 3k - 1$, any inflation of $\overline{K(n,k)}$ satisfies \nameref{conj:hc_chi} (\cref{lem:kneser_graphs2}).
              \item For integers $n,k,t \geq 1$ with $2k - t \leq n < \frac{5}{2} (k-t)$, any inflation $G$ of $K(n, k,\geq t+1)$ satisfies \nameref{conj:cdm} (\cref{thm:generalised_kneser_graphs2}).
          \end{itemize}
    \item Complements of Strongly Regular Triangle-Free Graphs (\cref{ss:srg}):
          \begin{itemize}
              \item Let $G$ be the complement of the Clebsch graph, or the complement of the Mesner graph, or the complement of the Gewirtz graph. Then any inflation of $G$ satisfies \nameref{conj:hc_chi} (\cref{lem:inflation_clebsch,lem:inflation_m22,cor:gewirtz}).
              \item The complement of the Higman-Sims graph satisfies \nameref{conj:shc_half} (\cref{thm:higman_sims}).
          \end{itemize}
    \item Complements of Eberhard Graphs (\cref{ss:eberhard}): The complement of the Eberhard graph with parameter $p$ for each prime $p$ with $p \equiv 11 \pmod{12}$ satisfies \nameref{conj:shc_chi} (\cref{thm:eberhard}).
\end{enumerate}

\section{Properties of Counterexamples}
\label{s:lit_review}
For $\mathfrak{C} \in \{\text{\nameref{conj:hc_chi}, \nameref{conj:hc_half}, \nameref{conj:hc_epsilon}, \nameref{conj:linear_cm}, \nameref{conj:4-CM}, \nameref{conj:shc_chi}, \nameref{conj:shc_half}, \nameref{conj:cdm}}\}$, this section describes properties of a potential counterexample to $\mathfrak{C}$. In doing so, we survey and generalise several classical results on the $\alpha(G) = 2$ case of Hadwiger's Conjecture. The existing results fall into three broad categories: minimum and minimal counterexamples (\cref{ss:minimal}), induced subgraphs (\cref{ss:induced}), and clique ratios (\cref{ss:seagulls}).

\subsection{Minimal and Minimum Counterexamples}
\label{ss:minimal}
Let $\mathfrak{C} \in \{\text{\nameref{conj:hc_chi}, \nameref{conj:hc_half}, \nameref{conj:4-CM}, \nameref{conj:shc_chi}, \nameref{conj:shc_half}, \nameref{conj:cdm}}\}$.
A graph $G$ is a \defn{minimal counterexample} to $\mathfrak{C}$ if $G$ is a counterexample to $\mathfrak{C}$, but any proper induced subgraph of $G$ satisfies $\mathfrak{C}$. A graph $G$ is a \defn{minimum counterexample} to $\mathfrak{C}$ if $G$ is a counterexample to $\mathfrak{C}$, but every graph $H$ with $|V(H)|
    < |V(G)|$ satisfies $\mathfrak{C}$.
\citet*{PST03} showed that if $G$ is a minimum counterexample to \nameref{conj:hc_chi}, then $G$ satisfies a long list of properties (see \cref{tab:minimal_properties}). This section establishes similar properties for minimal and minimum counterexamples for $\mathfrak{C} \in \{\text{\nameref{conj:hc_chi}, \nameref{conj:shc_chi}}\}$. The proof is based on ideas from \citep{PST03}. We use $\mathdefn{\omega(G)}$ to denote the \defn{clique number}, $\mathdefn{\mu(G)}$ to denote the \defn{matching number}, $\mathdefn{\kappa(G)}$ to denote the \defn{vertex-connectivity}, $\mathdefn{\delta(G)}$ to denote the \defn{minimum degree}, and $\mathdefn{\diam(G)}$ to denote the \defn{diameter} of $G$. For distinct $u, v \in V(G)$, let \defn{$\dist_G(u,v)$} denote the \defn{distance} between $u$ and $v$ in $G$. For a vertex $v \in V(G)$, let \defn{$N_G(v)$} denote the set of neighbours of $v$ in $G$. For a subset $S \subseteq V(G)$, let \defn{$G[S]$} be the subgraph of $G$ induced by $S$. Let \defn{$G - S$} be the graph obtained from $G$ by deleting vertices of $S$ from $G$. Let $\mathdefn{G - v}$ denote $G - \{v\}$, and for an edge $xy \in E(G)$, let $\mathdefn{G - xy}$ (resp. $\mathdefn{G + xy}$) be the graph obtained by deleting (resp. adding) $xy$ from $G$. Let $\mathdefn{\overline{G}}$ denote the \defn{complement} of $G$.

\begin{lem}
    \label{lem:chromatic_matching}
    For each graph $G$ with $\alpha(G) = 2$, $\chi(G) = |V(G)| - \mu(\overline{G})$.
\end{lem}
\begin{proof}
    $\chi(G) \geq |V(G)| - \mu(\overline{G})$: Let $k:= \chi(G)$ and colour $G$ with $k$ colours. Since $\alpha(G) = 2$, each colour class has $1$ or $2$ vertices. Suppose $\ell$ of the colour classes have size $2$ and $m$ of the colour classes have size $1$. Thus, $k = \ell + m$ and $|V(G)| = 2\ell + m = \ell + k$. The colour classes of size $2$ correspond to a matching in $\overline{G}$, so $\ell \leq \mu(\overline{G})$. However, $\ell = |V(G)| - k = |V(G)| - \chi(G)$, so $|V(G)| - \chi(G) \leq \mu(\overline{G})$ and $\chi(G) \geq |V(G)| - \mu(\overline{G})$.

    $\chi(G) \leq |V(G)| - \mu(\overline{G})$: Let $M$ be the maximum matching in $\overline{G}$, and let $\ell = |M| = \mu(\overline{G})$. Let $m$ be the number of vertices in $V(G) \setminus V(M)$, so $|V(G)|=  2 \ell + m$. Use $\ell$ colours to colour $V(M)$ so that the endpoints of each edge of $M$ share the same colour, and endpoints of distinct edges receive distinct colours. Use $m$ different colours for each vertex in $V(G) \setminus V(M)$. This gives a proper colouring of $V(G)$ using $\ell + m$ colours. Therefore, $\chi(G) \leq k = \ell + m = |V(G)| - \mu(\overline{G})$.
\end{proof}
Since $\mu(\overline{G})$ can be computed in polynomial time \citep{Edmonds65a}, \cref{lem:chromatic_matching} implies that one can find the chromatic number of a graph with independence number $2$ in polynomial time. A graph $G$ is \defn{vertex-critical} if $\chi(G - v) < \chi(G)$ for each $v \in V(G)$. For an integer $k \geq 1$, a vertex-critical graph with $\chi(G) = k$ is called \defn{$k$-critical}.

A graph is \defn{decomposable} if $V(G)$ can be partitioned into two non-empty subsets $A$ and $B$ such that each vertex of $A$ is adjacent to each vertex of $B$. Equivalently, $G$ is decomposable if the complement $\overline{G}$ is disconnected. We need the following deep theorem of \citet{gallai_kristische_1963} (see  \citet{stehlik_critical_2003} for a strengthening).

\begin{thm}[\citet{gallai_kristische_1963}]
    \label{thm:deep_theorem_of_gallai}
    Every vertex-critical graph $G$ with $|V(G)| < 2\chi(G) - 1$ is decomposable.
\end{thm}
The following results lead to further properties satisfied by a minimal counterexample.

\begin{lem}[Theorem 3.3, \citep{PST03}]
    \label{lem:first_seagull} Let $G$ be a graph with $\alpha(G) = 2$. Then $\had(G) \geq (\omega(G) + |V(G)|) / 3$. Further, for any integer $k \geq 1$, if $|V(G)| \geq 2k - 1$ and $\omega(G) \geq k - 2$, then $\had(G) \geq k$.
\end{lem}
\begin{lem}
    \label{lem:non_adj_vtx_crit}
    Let $G$ be a graph with $\alpha(G) = 2$, and suppose $G$ is $k$-critical with $|V(G)| = 2k - 1$. Then the graph $G'$ obtained by deleting two non-adjacent vertices $x$ and $y$  is $(k-1)$-critical.
\end{lem}

\begin{proof}
    Let $G' := G - x - y$.
    Since $G$ is $k$-critical, $\chi(G') \leq \chi(G - x) \leq k -1$. If $\chi(G') \leq k -2$, then by assigning a new colour to $x$ and $y$, we obtain a proper colouring of $G$ with $k-1$ colours, so $\chi(G) \leq k -1$, a contradiction. Hence, $\chi(G') = k-1$. We prove that $G'$ is $(k-1)$-critical. Suppose to the contrary. Then there is a vertex $u \in V(G')$ such that $\chi(G' -u) = k-1$. Since $G$ is vertex-critical, $\chi(G - u) = k-1$.
    Since $\chi(G' - u) = k - 1$ and $|V(G' - u)| = 2k - 4$, by \cref{lem:chromatic_matching}, $\mu(\overline{G' - u}) = 2k - 4 - (k-1) = k- 3$. Therefore, $\mu(\overline{G - u}) = k - 2$. By \cref{lem:chromatic_matching}, $\chi(G - u) = 2k - 2- (k-2) = k$. This contradicts the assumption that $G$ is $k$-critical.
\end{proof}

\begin{lem}
    \label{lem:equiv_triangle_free}
    The following are equivalent for a triangle-free graph $G$ with ${|V(G)| \geq 3}$:
    \begin{enumerate}[(\alph*)]
        \item $G$ is edge-maximal triangle-free (that is, for any edge $xy \in E(\overline{G})$,  $G + xy$ contains a triangle),
        \item $G$ has diameter $2$,
        \item $\overline{G}$ is edge-minimal with $\alpha(\overline{G}) = 2$ (that is, for any $xy \in E(\overline{G})$, $\alpha(\overline{G} - xy) = 3$),
        \item $\overline{G}$ does not contain a dominating edge.
    \end{enumerate}
\end{lem}
\begin{proof}
    (a) $\Rightarrow $ (b): Say $G$ is edge-maximal triangle-free. If $\diam(G) = 1$, then $G$ is complete, contradicting the assumption that $G$ is triangle-free since $|V(G)| \geq 3$. So $\diam(G)\neq 1$. Suppose for the sake of contradiction that $\diam(G) \geq 3$. Let $v$ and $w$ be vertices at distance at least $3$ in $G$. So $v$ and $w$ have no common neighbour. Then $G + vw$ is triangle-free, contradicting  the edge-maximality of $G$. Thus, $\diam(G) = 2$.

    (b) $\Rightarrow $ (c):
    Say $G$ has diameter 2. So $G$ is connected, and any two non-adjacent vertices in $G$ have a common neighbour.
    Since $G$ is connected, $|E(G)| \neq 0$, implying that $\alpha(\overline{G}) > 1$. Since $G$ is triangle-free, $\alpha(\overline{G}) = 2$. In $\overline{G}$, any two adjacent vertices $x$ and $y$ have a common non-neighbour $u \in V(G)$. Then $\alpha(\overline{G} - xy) = 3$, since $\{x,y,u\}$ is an independent set in $\overline{G} - xy$.

    (c) $\Rightarrow$ (d): Let $xy \in E(\overline{G})$. Since $\alpha(\overline{G}) = 2$ and $\alpha(\overline{G} - xy) = 3$, there is an independent set $S$ of size $3$ in $\overline{G} - xy$ containing $x$ and $y$. Let $u$ be the other element of $S$. In $\overline{G}$, $u$ is not adjacent to $x$ and $y$, and hence $xy$ is not a dominating edge.

    (d) $\Rightarrow$ (a): Since $\overline{G}$ does not contain a dominating edge, in $G$, any two non-adjacent vertices $x$ and $y$ have a common neighbour $u \in V(G)$. Then $G + xy$ contains a triangle $uxy$, which proves (a).
\end{proof}

\begin{thm}[Theorem 3.4, \citep{PST03}]
    \label{thm:connectivity}
    Every connected graph $G$ with $\alpha(G) = 2$ and $\kappa(G) \leq |V(G)| / 2$ has a non-empty connected dominating matching.
\end{thm}
For any graph $G$ with $\alpha(G) = 2$, the non-neighbours of each vertex induce a complete graph. Thus,
\begin{obs}
    \label{obs:non-neighbours}
    For every graph $G$ with $\alpha(G) = 2$, we have $\delta(G) \geq |V(G)| - 1 - \omega(G)$.
\end{obs}

\begin{thm}
    \label{thm:minimal_properties}
    Let $G$ be a graph with $\alpha(G) = 2$. With reference to \cref{tab:minimal_properties}:
    \begin{enumerate}[label=(\alph*),leftmargin=2em]
        \item If $G$ is a minimum counterexample to \nameref{conj:hc_chi}, then $G$ satisfies Properties \propref{prop:vertex_critical} -- \propref{prop:proper_minor_less_chi}.

        \item If $G$ is a minimal counterexample to \nameref{conj:hc_chi}, then $G$ satisfies Properties \propref{prop:vertex_critical} -- \propref{prop:ABC_sizes}.

        \item If $G$ is a minimum counterexample to \nameref{conj:shc_chi}, then $G$ satisfies Properties \propref{prop:vertex_critical} -- \propref{prop:C5_condition}.

        \item If $G$ is a minimal counterexample to \nameref{conj:shc_chi}, then $G$ satisfies Properties \propref{prop:vertex_critical} -- \propref{prop:C5_condition}.
    \end{enumerate}
\end{thm}
Note that (a) was shown by \citet{PST03}, which is implied by Properties \propref{prop:vertex_critical} -- \propref{prop:ABC_sizes} of (b), since a minimum counterexample to \nameref{conj:hc_chi} is a minimal counterexample to \nameref{conj:hc_chi}.

\input{minimal_ppties_table}

\begin{proof}
    First, we prove (d).
    Let $G$ be a minimal counterexample to \nameref{conj:shc_chi}. Recall that \defn{$\had_2(G)$} is the order of the largest complete graph model of $G$ such that each branch set has one or two vertices. So $\had_2(G) < \chi(G)$ but $\had_2(H) \geq \chi(H)$ for every proper induced subgraph $H$ of $G$.

    \begin{enumerate}[label={}, leftmargin=0em, itemsep=1.5pt]
        \item Property \propref{prop:vertex_critical}: Since $G$ is a minimal counterexample to \nameref{conj:shc_chi}, $\had_2(G) < \chi(G)$, and for each vertex $v \in V(G)$, $\had_2(G - v) \geq \chi(G - v)$. Since $\had_2(G - v) \leq \had_2(G)$, we have $\chi(G - v) \leq \had_2(G - v) \leq \had_2(G) < \chi(G)$, so $\chi(G - v) < \chi(G)$. Thus $G$ is vertex-critical.
        \item Property \propref{prop:not_decomposable}: Suppose $G$ is decomposable. Then $V(G) = A \cup B$, where $A \cap B = \varnothing$, $A \neq \varnothing$ and $B \neq \varnothing$, and each vertex of $A$ is adjacent to each vertex of $B$. Since $G$ is a minimal counterexample, $\had_2(G[A]) \geq \chi(G[A])$ and $\had_2(G[B]) \geq \chi(G[B])$. However, since each vertex of $A$ is adjacent to each vertex of $B$, $\had_2(G) \geq \had_2(G[A]) + \had_2(G[B])$ and $\chi(G) = \chi(G[A]) + \chi(G[B])$. Hence, $\had_2(G) \geq \chi(G)$, a contradiction.
        \item Property \propref{prop:size_equals_2chi_minus_1}: Since $\alpha(G) = 2$, $|V(G)| \leq 2 \chi(G)$ and $|V(G -v)| \leq 2 \chi(G - v)$ for each vertex $v \in V(G)$. If $\left|V(G)\right| = 2 \chi(G)$, then by Property \propref{prop:vertex_critical}, $
                  2\chi(G) - 1 = \left|V(G - v)\right| \leq 2\chi(G - v) = 2 (\chi(G) - 1) = 2\chi(G) - 2$,
              a contradiction. Therefore, $|V(G)| \leq 2 \chi(G) - 1$. If $|V(G)| < 2 \chi(G) - 1$, then by Property \propref{prop:vertex_critical} and \cref{thm:deep_theorem_of_gallai}, $G$ is decomposable, contradicting Property \propref{prop:not_decomposable}. Thus, $|V(G)| = 2\chi(G) - 1$.
        \item Property \propref{prop:subgraph_vertex_critical}: Follows from Property \propref{prop:size_equals_2chi_minus_1} and \cref{lem:non_adj_vtx_crit}.
        \item Property \propref{prop:complement_minus_x_pm}: Let $v \in V(G)$ and let $G' = G - v$. By Property \propref{prop:size_equals_2chi_minus_1}, $|V(G)| = 2\chi(G) - 1$. By Property \propref{prop:vertex_critical}, $\chi(G') = \chi(G) - 1$. By \cref{lem:chromatic_matching}, $
                  \mu(\overline{G}') = |V(G')| - \chi(G') = 2 \chi(G) - 2 - (\chi(G) - 1) = \chi(G) - 1 = \frac{1}{2} |V(G')|$,
              and hence $\overline{G}'$ has a perfect matching.
        \item Property \propref{prop:no_connected_dominating_matching}: Suppose $G$ has a non-empty connected dominating matching $M$ with $t \geq 1$ edges. Let $G' = G - V(M)$, which by Property \propref{prop:size_equals_2chi_minus_1} has $2(\chi(G)-t) - 1$ vertices. This implies $
                  \chi(G') \geq \frac{|V(G')|}{\alpha(G')} \geq \frac{2(\chi(G)-t) - 1}{2} = \chi(G) -t - \frac{1}{2} $.
              Since $\chi(G')$ is an integer, $\chi(G') \geq \chi(G) -t$. Since $G'$ is a proper induced subgraph of $G$, $\had_2(G') \geq \chi(G') \geq \chi(G) -t$. However, each edge of $M$ is adjacent to each vertex of $G'$, implying that $\had_2(G) \geq \had_2(G') + t$. Consequently, $\had_2(G) \geq \chi(G)$, a contradiction.
        \item Property \propref{prop:no_dominating_edges}: By Property~\propref{prop:no_connected_dominating_matching}, $G$ has no dominating edge. The result follows from \cref{lem:equiv_triangle_free}.
        \item Property \propref{prop:kappa_geq_chi}: By Property \propref{prop:no_connected_dominating_matching} and \cref{thm:connectivity}, $\kappa(G) > |V(G)|/2$. By Property \propref{prop:size_equals_2chi_minus_1}, $\kappa(G) > (2\chi(G) - 1)/2 = \chi(G) - 1/2$. Since $\kappa(G)$ is an integer, $\kappa(G) \geq \chi(G)$.
        \item Property \propref{prop:delta_geq_chi}: For every graph $G$, $\delta(G) \geq \kappa(G)$. By Property \propref{prop:kappa_geq_chi}, $\kappa(G) \geq \chi(G)$.
        \item Property \propref{prop:hamiltonian}: From Properties \propref{prop:size_equals_2chi_minus_1} and \propref{prop:delta_geq_chi}, $\delta(G)\geq \chi(G) = (|V(G)| + 1)/2$. By Dirac's Theorem \citep{Dirac52a}, $G$ is Hamiltonian.
        \item Property \propref{prop:factor_critical}: Let $v \in V(G)$ be any vertex. By Property \propref{prop:size_equals_2chi_minus_1}, $|V(G -v)|$ is even. By Property \propref{prop:hamiltonian}, $G - v$ has a Hamiltonian path $P$. We obtain a perfect matching of $G -v$ by taking alternating edges in $P$.
        \item Property \propref{prop:complement_diameter_2}: By Property \propref{prop:no_dominating_edges} and \cref{lem:equiv_triangle_free},
              $\diam(\overline{G}) = 2$.
        \item Property \propref{prop:B_nonempty}: Let $xy \in E(\overline{G})$. Let $B:= N_G(x) \cap N_G(y)$, $A:= N_G(x) \setminus N_G(y)$ and $C:= N_G(y) \setminus N_G(x)$. Since $\alpha(G) = 2$, $A$ and $C$ are cliques. Suppose $B = \varnothing$. Then $|A| \geq (|V(G)| - 2) / 2$ or $|C| \geq (|V(G)| - 2) / 2$. However, then $A \cup \{ x \}$ or $C \cup \{ y \}$ is a clique with at least $|V(G)|/2$ vertices. Therefore, $\had_2(G) \geq |V(G)|/2$. By Property \propref{prop:size_equals_2chi_minus_1}, $\had_2(G) \geq (2\chi(G) - 1)/2 = \chi(G) - 1/2$. Since $\had_2(G)$ is an integer, $\had_2(G) \geq \chi(G)$. However, then $G$ is not a counterexample to \nameref{conj:shc_chi}, a contradiction.
        \item Property \propref{prop:B_non_neighbours}: Let $xy \in E(\overline{G})$. By Property \propref{prop:no_dominating_edges}, $xb$ is not a dominating edge, and there is a vertex $c$ adjacent to neither $x$ nor $b$. Since $\alpha(G) = 2$, $cy \in E(G)$, and $c$ is the desired non-neighbour of $b$. Similarly, $yb$ is not a dominating edge, so there is a vertex $a \in N_G(x) \setminus N_G(y)$ not adjacent to $b$.
        \item Property \propref{prop:A_C_common_neighbour}: Let $xy \in E(\overline{G})$ and $a \in N_G(x) \setminus N_G(y)$ and $c \in N_G(y) \setminus N_G(x)$. If $a$ and $c$ have a common non-neighbour $b$, then $ac \in E(G)$ since $\alpha(G) = 2$. Conversely, if $ac \in E(G)$, then by Property \propref{prop:no_dominating_edges}, $ac$ is not dominating, so there is a common non-neighbour $b$ of $a$ and $c$. Since $\alpha(G) = 2$, $N_G(x) \setminus N_G(y)$ and $N_G(y) \setminus N_G(x)$ are cliques in $G$, so $b \notin (N_G(x) \setminus N_G(y)) \cup (N_G(y) \setminus N_G(x))$. Since $xy \notin E(G)$, $V(G) \setminus \{x, y\} = N_G(x) \cup N_G(y)$, and hence $b \in N_G(x) \cap N_G(y)$.
        \item Property \propref{prop:C5_condition}: Let $xy \in E(\overline{G})$. By Property \propref{prop:B_nonempty}, there is a vertex $b \in N_G(x) \cap N_G(y)$. By Property \propref{prop:B_non_neighbours}, there is a vertex $a \in N_G(x) \setminus N_G(y)$ not adjacent to $b$ and a vertex $c \in N_G(y) \setminus N_G(x)$ not adjacent to $b$. Since $b$ is a common non-neighbour of $a$ and $c$, by Property \propref{prop:A_C_common_neighbour}, $ac \in E(G)$. Then $xacyb$ is an induced $5$-cycle in $G$.
    \end{enumerate}
    This completes the proof of (d). A minimum counterexample to \nameref{conj:shc_chi} is a minimal counterexample to \nameref{conj:shc_chi}, and therefore Properties \propref{prop:vertex_critical} -- \propref{prop:C5_condition} also hold for minimum counterexamples to \nameref{conj:shc_chi}. This completes the proof of (c).
    By replacing $\had_2(G)$ with $\had(G)$, the same arguments show that Properties \propref{prop:vertex_critical} -- \propref{prop:C5_condition} hold for minimal and minimum counterexamples to \nameref{conj:hc_chi}. In what follows, we prove the remaining properties for (b). Let $G$ be a minimal counterexample to \nameref{conj:hc_chi}.
    \begin{enumerate}[label={}, leftmargin=0em, itemsep=1.5pt]
        \item Property \propref{prop:chi_at_least_7}: \citet{RST-Comb93} proved Hadwiger's Conjecture for $K_6$ minor-free graphs. Consequently, since $G$ is a counterexample to \nameref{conj:hc_chi}, $\chi(G) > \had(G) \geq 6$. Thus $\chi(G) \geq 7$. (Note that this argument fails for \nameref{conj:shc_chi}.)
        \item Property \propref{prop:kappa_at_least_7}: This follows from Properties \propref{prop:kappa_geq_chi} and \propref{prop:chi_at_least_7}.
        \item Property \propref{prop:omega_at_most_chi_minus_3}: Follows from Property \propref{prop:size_equals_2chi_minus_1} and \cref{lem:first_seagull} (with $k = \chi(G)$).
        \item Property \propref{prop:delta_geq_chi_plus_1}: By \cref{obs:non-neighbours},
              $\delta(G) \geq |V(G)| - 1 - \omega(G)$.
              By Properties \propref{prop:size_equals_2chi_minus_1} and \propref{prop:omega_at_most_chi_minus_3}:
              $|V(G)| - 1 - \omega(G) \geq 2 \chi(G) - 2 - (\chi(G) - 3) = \chi(G) + 1$.
              Hence, $\delta(G) \geq \chi(G) + 1$.
        \item Property \propref{prop:ABC_sizes}: Let $xy \in E(\overline{G})$. The set $(N_G(x) \setminus N_G(y)) \cup \{x\}$ is a clique in $G$, so by Property \propref{prop:omega_at_most_chi_minus_3}, $|(N_G(x) \setminus N_G(y)) \cup \{x\}| \leq \chi(G) - 3$. Thus, $|N_G(x) \setminus N_G(y)| \leq \chi(G) -4$. Similarly, $|N_G(y) \setminus N_G(x)| \leq \chi(G) -4$.
              We deduce that $|N_G(x) \cap N_G(y)| \geq 2\chi(G) - 3 - 2(\chi(G) - 4) = 5$. By Properties \propref{prop:B_nonempty} and \propref{prop:B_non_neighbours}, $N_G(x) \setminus N_G(y) \neq \varnothing$. Suppose $|N_G(x) \setminus N_G(y)| = 1$, and let $a \in N_G(x) \setminus N_G(y)$. Then by Property \propref{prop:B_non_neighbours}, each vertex $b \in N_G(x) \cap N_G(y)$ is not adjacent to $a$. $N_G(x) \cap N_G(y)$ are non-neighbours of $a$, and since $\alpha(G) = 2$, $N_G(x) \cap N_G(y)$ is a clique. By Property \propref{prop:delta_geq_chi_plus_1}, $|N_G(x)| \geq \chi(G) + 1$, so $|N_G(x) \cap N_G(y)| \geq \chi(G)$. This contradicts Property \propref{prop:omega_at_most_chi_minus_3}. Thus, $|N_G(x) \setminus N_G(y)| \geq 2$. Similarly, $|N_G(y) \setminus N_G(x)| \geq 2$. We deduce that $|N_G(x) \cap N_G(y)| \leq |V(G-x-y)| - 4 = 2\chi(G) - 3 - 4 = 2\chi(G) - 7$.
    \end{enumerate}
    This completes the proof of (b). It remains to show (a), which was shown by \citet{PST03}. We include the proof for completeness. Let $G$ be a minimum counterexample to \nameref{conj:hc_chi}. A graph $H$ is an \defn{induced minor} of $G$ if $H$ can be obtained from $G$ by a sequence of vertex deletions and edge contractions.
    \begin{obs}
        \label{obs:induced_minors}
        \phantom{linebreaker}
        \begin{itemize}
            \item $H$ is an induced minor of $G$ if and only if there is a model $\eta: V(H) \to 2^{V(G)}$ of $H$ in $G$ such that $uv \in E(H)$ if and only if there is an edge of $G$ between a vertex of $\eta(u)$ and a vertex of $\eta(v)$ \citep{InducedMinors24}. (This is called an \defn{induced minor model} in the literature.)
            \item If $H$ is an induced minor of $G$, then $\alpha(H) \leq \alpha(G)$. (If $\eta: V(H) \to 2^{V(G)}$ is an induced minor model of $H$ in $G$, given an independent set $S \subseteq V(H)$, for each $x \in S$, choose a vertex $v_x \in \eta(x)$. Then $\{v_x: x \in S\}$ is an independent set in $G$.)
            \item If $H'$ is a minor of $G$, then there is an induced minor $H$ of $G$ such that $H'$ is a subgraph of $H$ with $V(H) = V(H')$.
        \end{itemize}
    \end{obs}
    \begin{enumerate}[label={}, leftmargin=0em, itemsep=1.5pt]
        \item Property \propref{prop:edge_critical}: Let $xy \in E(G)$, and let $H$ be the graph obtained from $G$ by contracting $xy$ to a new vertex $z$. Then $H$ is an induced minor of $G$, and by \cref{obs:induced_minors}, $\alpha(H) \leq 2$. Since $G$ is a minimum counterexample to \nameref{conj:hc_chi}, $\had(H) \geq \chi(H)$ and $\had(G) < \chi(G)$. Since $\had(H) \leq \had(G)$, we have $\chi(H) \leq \had(H) \leq \had(G) < \chi(G)$, so $\chi(H) < \chi(G)$. Colour $H$ with $\chi(H)$ colours. We obtain a $\chi(H)$-colouring of $G - xy$ by colouring $x$ and $y$ (which are not adjacent) with the colour given to $z$. Consequently, $\chi(G - xy) < \chi(G)$. (Note that Property \propref{prop:edge_critical} may not be true for minimal counterexamples of \nameref{conj:hc_chi}, since the proof requires the contraction operation. Further Property \propref{prop:edge_critical} may not be true for minimum counterexamples to \nameref{conj:shc_chi}, since $\had_2(H) \leq \had_2(G)$ may not hold if $H$ is a minor of $G$. For example, $\had_2(C_7) = 2$, but $\had_2(K_3) = 3$.)
        \item Property \propref{prop:proper_minor_less_chi}: Let $H'$ be a proper minor of $G$. Suppose $|V(H')| < |V(G)|$.  By \cref{obs:induced_minors}, there is an induced minor $H$ of $G$ such that $H'$ is a subgraph of $H$ with $V(H) = V(H')$. Since $G$ is a minimum counterexample to \nameref{conj:hc_chi}, $\had(H) \geq \chi(H)$ and $\had(G) < \chi(G)$. Since $\had(H) \leq \had(G)$, we have $\chi(H') \leq \chi(H) \leq \had(H) \leq \had(G) < \chi(G)$.
              If $|V(H')| = |V(G)|$, then $H'$ is a subgraph of $G-xy$ for some edge $xy \in E(G)$. By Property \propref{prop:edge_critical}, $\chi(H') \leq \chi(G - xy) < \chi(G)$. In each case, $\chi(H') < \chi(G)$.
    \end{enumerate}
    This completes the proof of \cref{thm:minimal_properties}.
\end{proof}

\begin{thm}
    \label{thm:equiv}
    \begin{itemize}
        \item \nameref{conj:hc_chi} is equivalent to \nameref{conj:hc_half}. More precisely, any counterexample to \nameref{conj:hc_half} is a counterexample to \nameref{conj:hc_chi}, and any minimal counterexample to \nameref{conj:hc_chi} is a counterexample to \nameref{conj:hc_half} \citep{PST03}.
        \item \nameref{conj:shc_chi} is equivalent to \nameref{conj:shc_half}. More precisely, any counterexample to \nameref{conj:shc_half} is a counterexample to \nameref{conj:shc_chi}, and any minimal counterexample to \nameref{conj:shc_chi} is a counterexample to \nameref{conj:shc_half}.
        \item \nameref{conj:cdm} implies \nameref{conj:shc_chi}, which implies \nameref{conj:hc_chi}.
    \end{itemize}
\end{thm}
\begin{proof}
    \begin{itemize}
        \item Let $G$ be a minimal counterexample to \nameref{conj:hc_chi}. By Property \propref{prop:size_equals_2chi_minus_1} of \cref{thm:minimal_properties}, $|V(G)| = 2\chi(G) - 1$. Since $G$ is a counterexample to \nameref{conj:hc_chi}, $\had(G) < \chi(G) = (|V(G)| + 1)/2$. Since $|V(G)|$ is odd and $\had(G)$ is an integer, $\had(G) < |V(G)|/2$, so $G$ fails \nameref{conj:hc_half}. Conversely, if $G$ is a counterexample to \nameref{conj:hc_half}, then $\had(G) < |V(G)|/2 \leq \chi(G)$, so $G$ is a counterexample to \nameref{conj:hc_chi}.
        \item We obtain the conclusion by replacing $\had(G)$ with $\had_2(G)$ in the above argument.
        \item It is straightforward to see that \nameref{conj:shc_chi} implies \nameref{conj:hc_chi}. A minimal counterexample $G$ to \nameref{conj:shc_chi} is connected, since otherwise $G$ is the disjoint union of two complete graphs. By Property \propref{prop:no_connected_dominating_matching} of \cref{thm:minimal_properties}, $G$ does not have a non-empty connected dominating matching, and thus is a counterexample to \nameref{conj:cdm}. \qedhere
    \end{itemize}
\end{proof}

\subsection{Induced Subgraphs}
\label{ss:induced}
A graph $G$ is \defn{$H$-free} if $G$ does not contain $H$ as an induced subgraph. This section summarises literature on Hadwiger's Conjecture for $H$-free graphs.

\begin{thm}[Modified from Theorem 5.1 of \protect\citep*{PST03}]
    \label{thm:c5}
    A graph $G$ with $\alpha(G) = 2$ is $C_5$-free if and only if every connected induced subgraph $H$ of $G$ with $\alpha(H) = 2$ has a dominating edge.
\end{thm}
\citet*{PST03} instead proved that every $C_5$-free graph $G$ with $\alpha(G) = 2$ has $\had(G) \geq |V(G)|/ 2$. \cref{thm:c5} strengthens this result, as we now explain. Suppose $G$ is a $C_5$-free graph with $\had(G) < |V(G)|/2$ and $\alpha(G) = 2$. Thus, $G$ fails \nameref{conj:hc_chi}. Let $H$ be a minimal induced subgraph of $G$ that fails \nameref{conj:hc_chi}. By \cref{thm:equiv}, $H$ fails \nameref{conj:hc_half}, and by Property \propref{prop:no_dominating_edges} of \cref{thm:minimal_properties}, $H$ has no dominating edges. However, $H$ is $C_5$-free, so \cref{thm:c5} implies that $H$ has a dominating edge, a contradiction.

\begin{proof}
    ($\Rightarrow $): Suppose $H$ is a connected induced subgraph of $G$ with no dominating edge. Since $\alpha(H) = 2$ and $H$ is connected, there are two vertices $x, y$ such that $\dist_H(x,y) = 2$. Let $b \in N_H(x) \cap N_H(y)$. The edge $xb \in E(H)$ is not dominating, so there is a vertex $c$ adjacent to neither $x$
    nor $b$. Since $\alpha(H) = 2$ and $xy \notin E(H)$, $N_H(x) \cup N_H(y) = V(H)$. Therefore, $c \in N_H(y) \setminus N_H(x)$. Similarly, since $yb$ is not a dominating edge, there is a vertex $a \in N_H(x) \setminus N_H(y)$. If $ac \notin E(H)$, then $\{a,b,c\}$ is an independent set of size $3$, a contradiction. So $ac \in E(H)$. Then $xacyb$ is an induced $5$-cycle in $H$. Thus $G$ is not $C_5$-free.

    ($\Leftarrow $): $C_5$ is an induced subgraph of $G$ with no dominating edges.
\end{proof}
\noindent
\begin{minipage}{0.55\textwidth}
    Let \defn{$B_7$} be the graph shown in \cref{fig:b7}. \citet{PST03} extended \cref{thm:c5}:
    \begin{thm}[Modified\protect\footnotemark\ from Theorem 5.3 of \protect\citep*{PST03}]
        \label{thm:b7}
        Every $B_7$-free connected graph $G$ with ${\alpha(G) = 2}$ has a non-empty connected dominating matching.
    \end{thm}
\end{minipage}%
\hfill
\begin{minipage}{0.4\textwidth}
    \centering
    \includegraphics[scale=1]{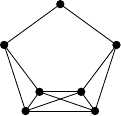}
    \captionof{figure}{The graph $B_7$.}
    \label{fig:b7}
\end{minipage}
\vspace{0.5em}
\footnotetext{For \cref{thm:b7}, \citet*{PST03} proved instead that every $B_7$-free graph $G$ with $\alpha(G) = 2$ has $\had(G) \geq |V(G)|/ 2$.}

This motivates the following definition.
Let $\mathfrak{C} \in \{\text{\nameref{conj:hc_chi}, \nameref{conj:hc_half}, \nameref{conj:shc_chi}, \nameref{conj:shc_half}, \nameref{conj:cdm}}\}$.
A graph $H$ with $\alpha(H) \leq 2$ is \defn{$\mathfrak{C}$-unavoidable} if every connected $H$-free graph $G$ with $\alpha(G) = 2$ satisfies $\mathfrak{C}$. Equivalently, $H$ is $\mathfrak{C}$-unavoidable if any counterexample to $\mathfrak{C}$ contains $H$ as an induced subgraph. In this language, \cref{thm:c5,thm:b7} say that $C_5$ and $B_7$ are \nameref{conj:cdm}-unavoidable. Note that if $H$ is an induced subgraph of $H'$ and $H'$ is $\mathfrak{C}$-unavoidable, then $H$ is $\mathfrak{C}$-unavoidable. By \cref{thm:equiv}, since \nameref{conj:cdm} implies \nameref{conj:shc_chi}, and \nameref{conj:shc_chi} implies \nameref{conj:hc_chi}, any \nameref{conj:cdm}-unavoidable graph is \nameref{conj:shc_chi}-unavoidable, and any \nameref{conj:shc_chi}-unavoidable graph is \nameref{conj:hc_chi}-unavoidable.
\begin{lem}
    Let $H$ be a graph with $\alpha(H) \leq 2$.
    \begin{itemize}
        \item $H$ is \nameref{conj:hc_chi}-unavoidable if and only if $H$ is \nameref{conj:hc_half}-unavoidable.
        \item $H$ is \nameref{conj:shc_chi}-unavoidable if and only if $H$ is \nameref{conj:shc_half}-unavoidable.
    \end{itemize}
\end{lem}
\begin{proof}
    Suppose $H$ is \nameref{conj:hc_chi}-unavoidable. Since \nameref{conj:hc_chi} implies \nameref{conj:hc_half}, $H$ is \nameref{conj:hc_half}-unavoidable. Conversely, suppose $H$ is \nameref{conj:hc_half}-unavoidable, and suppose $H$ is not \nameref{conj:hc_chi}-unavoidable. Let $G$ be an $H$-free graph with $\alpha(G) = 2$ that fails \nameref{conj:hc_chi}. Let $G'$ be a minimal induced subgraph of $G$ such that $G'$ is a counterexample to \nameref{conj:hc_chi}. Then $G'$ is $H$-free, and $\alpha(G') = 2$. Then $G'$ is a minimal counterexample to \nameref{conj:hc_chi}. By \cref{thm:equiv}, $G'$ fails \nameref{conj:hc_half}. This contradicts the assumption that $H$ is \nameref{conj:hc_half}-unavoidable. The analogous equivalence of \nameref{conj:shc_chi}-unavoidability and \nameref{conj:shc_half}-unavoidability follows a similar argument.
\end{proof}

One way to approach Hadwiger's Conjecture is to find $\mathfrak{C}$-unavoidable graphs. Progress in this direction is valuable since one can determine if a graph $G$ satisfies $\mathfrak{C}$ by examining its induced subgraphs. This, in turn, helps identify potential counterexamples. Adjacent-twin-free $\mathfrak{C}$-unavoidable graphs are especially useful, since \cref{lem:blow_up_avoidant} says that if $G$ is $H$-free and $H$ is adjacent-twin-free, then any inflation of $G$ is $H$-free.
\begin{lem}
    \label{lem:blow_up_avoidant}
    The following are equivalent for a graph $H$:
    \begin{enumerate}[(\alph*)]
        \item $H$ is adjacent-twin-free,
        \item $\overline{H}$ is twin-free,
        \item For any graph $G$, if $G$ is $H$-free, then any inflation $G'$ of $G$ is $H$-free,
        \item For any graph $G$, if $G$ is $\overline{H}$-free, then any blow-up $G'$ of $G$ is $\overline{H}$-free.
    \end{enumerate}
\end{lem}
\begin{proof}
    The equivalence of (a) and (b) is straightforward. The equivalence of (c) and (d) is straightforward. We show the equivalence of (b) and (d).

    (d) $\Rightarrow $ (b): Suppose $\overline{H}$ is not twin-free. Then there exists $uv \notin E(\overline{H})$ with $N_{\overline{H}}(u) = N_{\overline{H}}(v)$. Consider $G := \overline{H} - v$. Then $G$ is $\overline{H}$-free. However, since $N_{\overline{H}}(u) = N_{\overline{H}}(v)$, $\overline{H}$ is a blow-up of $G$ with $u \in V(\overline{H})$ replaced by an independent set of size $2$. This contradicts (d).

    (b) $\Rightarrow$ (d): Suppose (d) does not hold. Then there exists a graph $G$ and a blow-up $G'$ of $G$ such that $G$ is $\overline{H}$-free, but $G'$ is not $\overline{H}$-free. Let $p: V(G') \to V(G)$ be the projection map. Consider an induced subgraph $J$ of $G'$ isomorphic to $\overline{H}$. Since $G$ is $\overline{H}$-free, there are two non-adjacent vertices $v, v' \in V(J)$ with $p(v) = p(v')$; that is, $N_{G'}(v) = N_{G'}(v')$. However, $N_{\overline{H}}(v) = N_{G'}(v) \cap V(\overline{H}) = N_{G'}(v') \cap V(\overline{H}) = N_{\overline{H}}(v')$, so $\overline{H}$ is not twin-free.
\end{proof}

\begin{lem}
    \label{lem:unavoidable_adj_twin_free}
    Let $\mathfrak{C} \in \{\text{\nameref{conj:hc_chi}, \nameref{conj:hc_half}, \nameref{conj:shc_chi}, \nameref{conj:shc_half}, \nameref{conj:cdm}}\}$, and suppose $H$ is an adjacent-twin-free $\mathfrak{C}$-unavoidable graph. If $G$ is a connected $H$-free graph with $\alpha(G) = 2$, then any inflation of $G$ satisfies $\mathfrak{C}$.
\end{lem}
\begin{proof}
    Since $H$ is adjacent-twin-free, by \cref{lem:blow_up_avoidant} and the assumption that $G$ is $H$-free, any inflation of $G$ is $H$-free. Since $H$ is $\mathfrak{C}$-unavoidable, any inflation of $G$ satisfies $\mathfrak{C}$.
\end{proof}
\input{forbidden_induced}
\citet{Kriesell10} conjectured there are infinitely many \nameref{conj:hc_chi}-unavoidable graphs. This conjecture, if true, would imply that any counterexample to \nameref{conj:hc_chi} contains an infinite number of graphs as induced subgraphs, consequently implying \nameref{conj:hc_chi}. Except for the $4$-cycle $C_4$, $B_7$ contains every $4$-vertex graph with independence number at most $2$ as an induced subgraph. \citet{PST03} showed that $C_4$ is \nameref{conj:cdm}-unavoidable. Combined with \cref{thm:b7}, this implies that each graph $H$ with $|V(H)| \leq 4$ and $\alpha(H) \leq 2$ is \nameref{conj:cdm}-unavoidable. \citet{Kriesell10} extended this result, proving that each graph $H$ with $|V(H)| \leq 5$ and $\alpha(H) \leq 2$ is \nameref{conj:cdm}-unavoidable. It is not known if every graph on six vertices is \nameref{conj:shc_chi}-unavoidable (or even \nameref{conj:hc_chi}-unavoidable).
For \nameref{conj:hc_chi}-unavoidable graphs, let \defn{$W_5$} denote the $5$-wheel, obtained from the $5$-cycle by adding a vertex adjacent to each vertex of the $5$-cycle. \citet{Bosse19} proved that $W_5$ is \nameref{conj:hc_chi}-unavoidable.
Using ideas from \citet{Bosse19}, \citet{Carter22} computationally proved that there is a list $H_1, \dots, H_{33}$ of \nameref{conj:hc_chi}-unavoidable graphs, none of which contains another as an induced subgraph.
The complements of $H_1, \dots, H_{33}$ are shown in \cref{tab:unavoidable_graphs}.

Furthermore, \citet{Carter22} proved that $K_8$ is \nameref{conj:hc_chi}-unavoidable, and
\citet{Zhou23} proved that $\overline{K_{1,6}}$ is \nameref{conj:hc_chi}-unavoidable\footnote{\citet{Zhou23} actually proved that a different graph is \nameref{conj:hc_chi}-unavoidable (see Theorem 1.6, \citep{Zhou23}); however, that graph is an induced subgraph of one of \citet{Carter22}'s $33$ graphs.}.
The joint results in \citep{Bosse19, Carter22, Zhou23,PST03} imply that each graph from \cref{tab:unavoidable_graphs} is \nameref{conj:hc_chi}-unavoidable. Equivalently, if $G$ is a counterexample to \nameref{conj:hc_chi}, then $G$ contains every graph from \cref{tab:unavoidable_graphs} as an induced subgraph.

\subsection{Seagulls, Inflations and Clique Ratios}
\label{ss:seagulls}

This section summarises the literature on Hadwiger's Conjecture for graphs with large \defn{clique ratio}, defined for a graph $G$ to be $\omega(G) / |V(G)|$.
A \defn{seagull} $S$ of a graph $G$ is an induced $3$-vertex-path of $G$. If $\alpha(G) = 2$, then each vertex of $G - S$ is adjacent to $S$.
Thus, if $G$ has $k$ vertex-disjoint seagulls, then $\had(G) \geq k$. \citet{CS12} proved a necessary and sufficient condition for a graph $G$ with $\alpha(G) = 2$ to have $k$ vertex-disjoint seagulls.

\begin{thm}[\citet{CS12}]
    \label{thm:seagulls}
    Let $k \geq 1$ be an integer and let $G$ be a graph with $\alpha(G) = 2$ and $G \ncong W_5$. Then $G$ has $k$ vertex-disjoint seagulls if and only if:
    \begin{enumerate}
        \item $|V(G)| \geq 3k$,
        \item $\kappa(G) \geq k$,
        \item For each clique $C \subseteq V(G)$, if $D$ is the set of vertices in $V(G) \setminus C$ that have both a neighbour and a non-neighbour in $C$, then $|D| + |V(G) \setminus C| \geq 2k$, and
        \item $\mu(\overline{G}) \geq k$.
    \end{enumerate}
\end{thm}
There are several easier-to-use results that arise from \cref{thm:seagulls}. For example, \citet{NS26a} showed that for any integer $k\geq 1$, if a graph $G$ has $\alpha(G) = 2$ and $|V(G)| \geq 3k$ and $\omega(G) \leq k$, then $G$ has $k$ pairwise vertex-disjoint seagulls. As another example \citep{CS12}, for $t := \ceil{|V(G)| / 2}$, if $G$ is $t$-connected and the largest clique $Z$ in $G$ satisfies $\frac{3}{2} \left \lceil \frac{|V(G)|}{2} \right \rceil - \frac{|V(G)|}{2} \leq |Z| \leq t$, then there are $t-|Z|$ pairwise vertex-disjoint seagulls in $V(G) \setminus Z$, and consequently $G$ has a $K_t$ minor, and $G$ satisfies \nameref{conj:hc_half}. Another application of \cref{thm:seagulls} is:
\begin{samepage}
    \begin{thm}[\citet{CS12}]
        \label{thm:seagullsN/4}
        Every graph $G$ with $\alpha(G) = 2$ and
        \begin{equation*}
            \omega(G) \geq
            \begin{cases}
                \left\lceil\frac{|V(G)|}{4}\right\rceil   & \text{if $|V(G)|$ is even,} \\
                \noalign{\vskip1pt}
                \left\lceil\frac{|V(G)|+3}{4}\right\rceil & \text{if $|V(G)|$ is odd,}
            \end{cases}
        \end{equation*}
        has $\had(G) \geq |V(G)|/2$, and consequently $G$ satisfies \nameref{conj:hc_half}.
    \end{thm}
\end{samepage}

\cref{thm:seagullsN/4} has been misunderstood by several authors (\citep[Theorem 1.6]{Bosse19}, \citep[Lemma 2]{Carter22}, \citep[Theorem 2.2]{Zhou23}). They wrongly conclude that $G$ satisfies \nameref{conj:hc_chi} instead of \nameref{conj:hc_half}. This is not implied.

\cref{thm:seagullsN/4} says that graphs with clique ratio roughly at least $1/4$ satisfy \nameref{conj:hc_half}, and are not minimal counterexamples to \nameref{conj:hc_chi} by \cref{thm:equiv}. For a graph $G$, the clique ratios of inflations of $G$ are key to proving that any inflation of $G$ satisfies \nameref{conj:hc_half}. This motivates the following definition. Define the \defn{inflated clique ratio} of a non-empty graph $G$ as:

\begin{equation*}
    \mathdefn{\icr(G)} = \inf \left \{\frac{\omega(G')}{|V(G')|} : \text{$G'$ is an inflation of $G$ with $|V(G')| > 0$}\right \}.
\end{equation*}
For every non-empty graph $G$, $0 \leq \icr(G) \leq \omega(G) / |V(G)| \leq 1$, so $\icr(G)$ is finite and non-negative. If $\icr(G)$ is roughly at least $1/4$, then any inflation of $G$ satisfies \nameref{conj:hc_half}. We thus turn our attention to providing lower bounds on $\icr(G)$.
The inflated clique ratio is related to the fractional clique covering number of $G$, which we now introduce.

For an integer $b \geq 1$, a \defn{$b$-fold colouring} of a graph $G$ is an assignment of sets of size $b$ to each vertex of $G$ such that adjacent vertices are assigned disjoint sets. For an integer $a \geq b$, an \defn{$(a:b)$-colouring} is a $b$-fold colouring using at most $a$ colours in total. The \defn{$b$-fold chromatic number}, denoted \defn{$\chi_b(G)$}, is the least integer $a$ such that an $(a:b)$-colouring exists. The \defn{fractional chromatic number} of $G$ is defined as:
\begin{equation*}
    \mathdefn{\chi_f(G)}:= \lim_{b \to \infty} \frac{\chi_b(G)}{b} = \inf_b \frac{\chi_b(G)}{b}.
\end{equation*}
The \defn{fractional clique covering number} of $G$, denoted $\mathdefn{\theta_f(G)}$, is defined to be $\chi_f(\overline{G})$.

A \defn{clique cover} of $G$ is a partition of $V(G)$ into a collection of cliques. The minimum number of cliques needed to cover $G$ is known as the \defn{clique covering number} of $G$, denoted $\mathdefn{\theta(G)}$. Note that $\theta(G) = \chi(\overline{G})$. The following inequalities are well-known:
\begin{obs}
    \label{obs:fractional}
    For any graph $G$,
    \begin{enumerate}
        \item $\omega(\overline{G}) \leq \theta_f(G) = \chi_f(\overline{G}) \leq \chi(\overline{G})$,
        \item $\theta_f(G) \leq \theta(G)$, and
        \item $\chi_f(\overline{G}) \alpha(\overline{G}) \geq |V(\overline{G})|$, and $\theta_f(G) \omega(G) \geq |V(G)|$.
    \end{enumerate}
\end{obs}
In \cref{lem:inflation_frac}, we prove that proper inflations preserve the fractional clique covering number. \cref{lem:equiv_frac} is a key step in the proof. Let $\mathdefn{\Q_{\geq 1}}:=\{q \in \Q: q \geq 1\}$.

\begin{lem}
    \label{lem:equiv_frac}
    Let $k \in \Q_{\geq 1}$ and let $G$ be a graph. Then $\theta_f(G) \leq k$ if and only if there are (not necessarily distinct) cliques $X_1, \dots, X_r$ such that each vertex $v \in V(G)$ belongs to at least $r/k$ of $X_1, \dots, X_r$.
\end{lem}
\begin{proof}
    ($\Rightarrow$): Suppose $k = a/b$, where $a, b \geq 1$ are coprime integers. Suppose $\theta_f(G) = \chi_f(\overline{G}) = r/d$, where $r,d \geq 1$ are coprime integers, and $r \geq d$. For each vertex $v \in V(G)$, let $S_v$ be the set of $d$ colours assigned to $v$ out of the colours $\{1, \dots, r\}$ in the $(r:d)$-colouring of $\overline{G}$. For each $i \in \{1, \dots, r\}$, let $X_i:= \{v \in V(G) : i \in S_v\}$. Since adjacent vertices in $\overline{G}$ receive disjoint sets, each $X_i$ is an independent set in $\overline{G}$ and hence a clique in $G$. Since $|S_v| = d$, each vertex is in exactly $d$ of the cliques $X_1, \dots, X_r$. By assumption, $\theta_f(G) = r/d \leq a/b = k$, so $rb \leq ad$ and $d \geq br/a = r/k$, implying that each vertex $v$ is in at least $r/k$ cliques, as desired.

    ($\Leftarrow$): Let $X_1, \dots, X_r$ be cliques such that each vertex $v \in V(G)$ is in at least $r/k$ cliques. Suppose $k = a/b$, where $a, b \geq 1$ are coprime integers. For each $j \in \{1, \dots, r\}$, create $a$ copies of $X_j$ by defining, for each $i \in \{1, \dots, a\}$, $X_j^{(i)}:= X_j$. Then
    $\mathcal{X}:= (X_1^{(1)}, \dots, X_1^{(a)}, \dots, X_r^{(1)}, \dots, X_r^{(a)})$
    is a multiset of $ra$ cliques, where each vertex $v \in V(G)$ is in at least $ra/k = rb$ cliques in $\mathcal{X}$. Assign each clique in $\mathcal{X}$ a distinct colour from $\{1, \dots, ra\}$, and for each vertex $v \in V(G)$, assign $v$ a set $S'_v$ of colours such that the colour $c \in \{1, \dots, ra\}$ is in $S'_v$ if and only if $v$ is contained in the clique of $\mathcal{X}$ coloured $c$. Observe that $|S'_v| \geq rb$. For distinct vertices $v_1, v_2 \in V(G)$, if $S_{v_1}' \cap S_{v_2}' \neq \varnothing$, then $v_1,v_2 \in X_i$ for some $i \in \{1, \dots, r\}$, and $v_1v_2 \notin E(\overline{G})$. In $\overline{G}$, this defines a colouring using $ra$ colours where adjacent vertices receive disjoint sets. For each $v \in V(G)$, choose a subset $S_v \subseteq S'_v$ such that $|S_v| = rb$. This defines a $(ra: rb)$-colouring of $\overline{G}$, so $\theta_f(G) \leq ra/rb = k$.
\end{proof}

\begin{lem}
    \label{lem:inflation_frac}
    Any proper inflation $G'$ of a graph $G$ satisfies $\theta_f(G) = \theta_f(G')$.
\end{lem}
\begin{proof}
    $\theta_f(G') \leq \theta_f(G)$: Set $k:=\theta_f(G)$. \cref{lem:equiv_frac} implies that there are cliques $X_1, \dots, X_{r}$ such that each vertex $v \in V(G)$ is in at least $r/k$ cliques of $X_1, \dots, X_{r}$. For each $i \in \{1, \dots, r\}$, define $
        X_i':= \bigcup_{v \in X_i} V(C_v)$,
    where $C_v$ is the clique in $G'$ that replaces $v$. Each vertex of $G'$ is contained in at least $r/k$ of $X_1', \dots, X_r'$. \cref{lem:equiv_frac} implies that $\theta_f(G') \leq k$.

    $\theta_f(G) \leq \theta_f(G')$: Set $k:=\theta_f(G')$. \cref{lem:equiv_frac} implies that there are cliques $X_1', \dots, X_r'$ such that each vertex $v \in V(G')$ is in at least $r/k$ of $X_1', \dots, X_r'$. Since $G'$ is a proper inflation of $G$, the projection map $p: V(G') \to V(G)$ is surjective, and projects cliques of $G'$ onto (not necessarily distinct) cliques of $G$. For each $i \in \{1, \dots, r\}$, let $X_i:= p(X_i')$ be the projected clique in $G$. Since $p$ is surjective, for each vertex $v \in V(G)$, there is a vertex $v' \in V(G')$ such that $p(v') = v$. By assumption, $v'$ is in at least $r/k$ of $X_1', \dots, X_r'$. This implies that $v$ is in at least $r/k$ of $X_1, \dots, X_r$. \cref{lem:equiv_frac} implies $\theta_f(G) \leq k$.
\end{proof}
Since $\theta_f(G)$ does not increase under taking induced subgraphs, the following is an immediate consequence of \cref{lem:inflation_frac}.
\begin{cor}
    \label{cor:inflation_frac}
    Any inflation $G'$ of a graph $G$ satisfies $\theta_f(G') \leq \theta_f(G)$.
\end{cor}
\cref{lem:jofre_lemma} follows from \cref{obs:fractional,lem:equiv_frac}.
\begin{lem}
    \label{lem:jofre_lemma}
    Let $X_1, \dots, X_r$ be cliques of a graph $G$ such that each vertex $v \in V(G)$ appears in at least $k$ of $X_1, \dots, X_r$. Then $\omega(G)/ |V(G)| \geq k/r$.
\end{lem}
\cref{cor:inflation_frac}, \cref{lem:jofre_lemma}, and the definition of $\icr(G)$ imply that:
\begin{lem}
    \label{lem:clique_cover}
    For every graph $G$,
    \begin{equation*}
        \icr(G) \geq \frac{1}{\theta_f(G)} \geq \frac{1}{\theta(G)}.
    \end{equation*}
\end{lem}
For upper bounds on $\icr(G)$, note that for every graph $G$, $\icr(G) \leq 1/\alpha(G)$. To see this, let $I \subseteq V(G)$ be a maximum independent set, and let $G' := G[I]$. Then $G'$ has no edges, so $\omega(G') = 1$, but $|V(G')| = \alpha(G)$. Therefore, $\icr(G) \leq 1/\alpha(G)$. For perfect graphs, equality holds. If $G$ is perfect, then by
the Weak Perfect Graph Theorem \citep[Proposition~5.5.3]{Diestel05}, $\overline{G}$ is perfect, so $\theta(G) = \chi(\overline{G}) = \omega(\overline{G}) = \alpha(G)$. The result follows from \cref{lem:clique_cover}.

\cref{conj:inflation,cor:inflation_frac} motivate a natural question.
\begin{quest}
    \label{q:fractional}
    Let $\mathfrak{C} \in \{\text{\nameref{conj:hc_chi}, \nameref{conj:hc_half}, \nameref{conj:shc_chi}, \nameref{conj:shc_half}, \nameref{conj:cdm}}\}$.
    For which values of $k \in \Q_{\geq 1}$ does the following hold: Every graph with $\alpha(G) = 2$ and $\theta_f(G) \leq k$ satisfies $\mathfrak{C}$.
\end{quest}
The next result, due to \citet{Blasiak07}, is the first result in the direction of \cref{q:fractional}. For a graph $G$ with $\theta_f(G) < 3$, by \cref{lem:equiv_frac}, there are cliques $X_1, \dots, X_r$ of $G$ such that each vertex of $G$ belongs to strictly more than $r/3$ of $X_1, \dots, X_r$. If $S \subseteq V(G)$ is an independent set of size $3$, then by the Pigeonhole Principle, at least two vertices in $S$ are in some clique $X_i$, contradicting $S$ being an independent set. Thus, every graph with $\theta_f(G) < 3$ has $\alpha(G) \leq 2$.
\begin{thm}[\citet{Blasiak07}]
    \label{thm:bla}
    Every graph $G$ with $\theta_f(G) < 3$ and $|V(G)|$ even satisfies \nameref{conj:shc_half}.
\end{thm}
Note that \cref{thm:bla} does not imply that every graph $G$ with $\theta_f(G) < 3$ satisfies \nameref{conj:shc_chi}. \cref{thm:bla} is a special case of a more general result, \cref{thm:good_bad}.
\begin{thm}[\citet{Blasiak07}]
    \label{thm:good_bad}
    Let $G$ be a graph with $\alpha(G) = 2$ and $|V(G)|$ even. Suppose there is a partition of $E(G)$ into `good' and `bad' edges such that:
    \begin{enumerate}[label=(\arabic*)]
        \item any two good edges that do not share an endpoint are adjacent, and
        \item for any two bad edges $uv$ and $vw$ that share an endpoint, $uw \in E(G)$, and
        \item if $u$ and $v$ are adjacent-twins then $uv$ is bad, and
        \item subject to (3), the number of good edges is as large as possible.
    \end{enumerate}
    Then either:
    \begin{enumerate}[label=(\alph*)]
        \item $G$ has a dominating edge, or
        \item $\kappa(G) \leq |V(G)|/2$, or
        \item $\omega(G) \geq |V(G)|/2$, or
        \item $G$ has a perfect matching $M$ of good edges such that $M$ is connected (and hence $M$ is dominating).
    \end{enumerate}
\end{thm}
We deduce \cref{thm:bla} from \cref{thm:good_bad}.
Note that if (b), (c) or (d) occurs, then $G$ satisfies \nameref{conj:shc_half}. If (a) occurs, then we deduce that $G$ satisfies \nameref{conj:shc_half} by applying induction to $G - \{u, v\}$, where $uv$ is the dominating edge of $G$.
\cref{thm:good_bad} implies \cref{thm:bla} because for every graph with $\theta_f(G) < 3$, there is a partition of $E(G)$ into good and bad edges. To see this, observe that by \cref{lem:equiv_frac}, there is a multiset $\mathcal{X}:=(X_1, \dots, X_r)$ in $G$ such that each vertex of $G$ is in strictly more than $r/3$ cliques of $\mathcal{X}$. Partition $E(G)$ as follows. For an edge $uv \in E(G)$, let $\mathdefn{n(uv)}$ be the size of the multiset $\mathdefn{N(uv)}:= (X_i \in \mathcal{X}: \text{$u$ or $v$ is in $X_i$})$. An edge $uv$ is good if $n(uv) > r/2$, and bad otherwise. If $uv$ and $xy$ are good edges not sharing an endpoint, then since $n(uv), n(xy) > r/2$, we have $N(uv) \cap N(xy) \neq \varnothing$, so there is a clique in $\mathcal{X}$ containing an endpoint of $uv$ and an endpoint of $xy$. Thus, two good edges that do not share an endpoint are adjacent. If $uv$ and $vw$ are bad, then $n(uv) \leq r/2$ and $n(vw) \leq r/2$. However, since there are strictly more than $r/3$ cliques containing $u$, there are strictly less than $r/6$ cliques of $\mathcal{X}$ containing $v$ but not $u$. Similarly, there are strictly less than $r/6$ cliques of $\mathcal{X}$ containing $v$ but not $w$. However, since $v$ is in strictly more than $r/3$ cliques of $\mathcal{X}$, there is a clique in $\mathcal{X}$ containing $u,v$ and $w$. Thus, $uw \in E(G)$.

It is tempting to conjecture that every graph $G$ with $\alpha(G) = 2$ has a partition of $E(G)$ into good and bad edges satisfying the assumptions in \cref{thm:good_bad}. However, \citet{Blasiak07} demonstrated that this is false.

If $|V(G)| = 2k - 1$ is odd, then \cref{thm:bla} implies that $\had_2(G) \geq k-1$, instead of the desired $\had_2(G) \geq k$. Despite a fair amount of effort \citep{CS12}, the following conjecture is open.
\begin{conj}
    \label{conj:bla}
    Every graph $G$ with $\theta_f(G) < 3$ satisfies \nameref{conj:shc_half}.
\end{conj}
By allowing branch sets of size $3$, we can remove the assumption that $|V(G)|$ is even and improve the bound on $\theta_f(G)$ in \cref{thm:bla} (see \cref{thm:124/33}). We start with the following result, proved by \citet{Carter22}.
\begin{thm}[\citet{Carter22}]
    \label{thm:carter30}
    Let $G$ be a graph with $\alpha(G) = 2$ and $|V(G)| \leq 30$. Then $G$ satisfies \nameref{conj:hc_chi}. Therefore, any minimal counterexample to \nameref{conj:hc_chi} has at least $31$ vertices, and $|V(G)|$ is odd.
\end{thm}
\begin{proof}
    Let $G$ be a minimal counterexample to \nameref{conj:hc_chi}. By \cref{thm:equiv}, $G$ fails \nameref{conj:hc_half}.
    Since     $K_8$ is \nameref{conj:hc_chi}-unavoidable~\citep{Carter22}, $\omega(G) \geq 8$. If $|V(G)| \leq 30$, then the assumptions of \cref{thm:seagullsN/4} are satisfied, so $G$ satisfies  \nameref{conj:hc_half}, a contradiction. Therefore, $|V(G)| \geq 31$. By Property \propref{prop:size_equals_2chi_minus_1} of \cref{thm:minimal_properties}, $|V(G)|$ is odd.
\end{proof}
\begin{thm}
    \label{thm:icr1/4}
    Let $G$ be a graph with $\alpha(G) = 2$ and $\icr(G) \geq \frac{1}{4}  + \varepsilon$, where $\varepsilon:=1/62$. Then any inflation of $G$ satisfies \nameref{conj:hc_chi}.
\end{thm}
\begin{proof}
    Let $G'$ be an inflation of $G$ that fails \nameref{conj:hc_chi}. Let $H$ be a minimal induced subgraph of $G'$ that is a counterexample to \nameref{conj:hc_chi}. Note that $H$ is also an inflation of $G$, and $H$ is a minimal counterexample to \nameref{conj:hc_chi}.
    By \cref{thm:carter30}, $|V(H)| \geq 31$ and $|V(H)|$ is odd. By  \cref{thm:equiv}, $H$ fails \nameref{conj:hc_half}.

    Since $\icr(G) \geq \frac{1}{4}  + \varepsilon$, $
        \omega(H) \geq \left\lceil \frac{|V(H)|}{4} + \frac{|V(H)|}{62}\right \rceil \geq \left\lceil \frac{|V(H)| + 2}{4} \right\rceil$.
    If $|V(H)| \equiv 1 \pmod 4$, then $|V(H)| = 4k + 1$ for some integer $k$. Then $
        \omega(H) \geq \left\lceil \frac{(4k+1)+ 2}{4} \right\rceil = k + 1$.
    By \cref{thm:seagullsN/4}, since $\omega(H) \geq \left\lceil \frac{|V(H)| + 3}{4} \right\rceil = k + 1$, $H$ satisfies \nameref{conj:hc_half}, a contradiction.
    If $|V(H)| = 4k + 3$ for some integer $k$, then $
        \omega(H) \geq \left\lceil \frac{(4k + 3) + 2}{4} \right\rceil = k + 2$.
    By \cref{thm:seagullsN/4}, since $\omega(H) \geq \left\lceil \frac{|V(H)| + 3}{4} \right\rceil = k + 2$, $H$ satisfies \nameref{conj:hc_half}, a contradiction.
\end{proof}

\begin{thm}
    \label{thm:124/33}
    Let $G$ be a graph with $\alpha(G) = 2$ and $\theta_f(G) \leq 124/33 \approx 3.7575\dots$. Then any inflation of $G$ satisfies \nameref{conj:hc_chi}.
\end{thm}
\begin{proof}
    By \cref{lem:clique_cover}, $\icr(G) \geq \frac{33}{124} = \frac{1}{4} + \frac{1}{62}$. The result follows from \cref{thm:icr1/4}.
\end{proof}
In some sense, \cref{thm:seagullsN/4} should allow us to push the constant $124/33$ in \cref{thm:124/33} up to $4$. After all, every graph with $\theta_f(G) \leq 4$ has $\omega(G)/|V(G)| \geq 1/4$, and \cref{thm:seagullsN/4} says that \nameref{conj:hc_half} holds for graphs with $\omega(G) / |V(G)|$ slightly larger than $1/4$.
We are unable to push the constant $124/33$ up to $4$ because of the $\varepsilon$ term in \cref{thm:icr1/4}, which depends on the lower bound on the size of a minimum counterexample established in \cref{thm:carter30}.
We set out to prove a qualitative strengthening of \cref{thm:124/33}, introduced by \citet{Carter22}\footnote{In \citet{Carter22}, this is called a `four clique cover'. We avoid the term `clique cover', since the definition typically requires the cliques to be disjoint.}, which turns out to be useful for handling inflations.
\begin{lem}
    \label{lem:cc2}
    Let $G$ be a graph, and let $\mathcal{K}$ be a collection of cliques of $G$ whose union covers $V(G)$. Then any proper inflation $G'$ of $G$ has a collection of cliques $\mathcal{K}'$ whose union covers $V(G')$ such that $|\mathcal{K}| = |\mathcal{K}'|$ and $\sum_{Q \in \mathcal{K'}} |Q| = |V(G')| - |V(G)| + \sum_{Q \in \mathcal{K}} |Q|$.
\end{lem}
\begin{proof}
    Let $G'$ be a proper inflation of $G$. We induct on $|V(G')|$. Let $V(G) = \{ v_1, \dots, v_n \}$.
    The base case when $|V(G')| = |V(G)|$ is straightforward by taking $\mathcal{K}' = \mathcal{K}$.

    For the induction step, suppose $G'$ is an inflation of $G$ with inflation sizes $k_1, \dots, k_n$, with $k_i \geq 1$ for each $i \in \{ 1, \dots, n \}$. Since $|V(G')| > |V(G)|$, there is an index $i$ such that $k_i \geq 2$. Let $G''$ be an inflation of $G$ with inflation sizes $k_1, \dots k_{i-1},  k_i - 1, k_{i+1} \dots,k_n$. Then $G''$ is a proper inflation of $G$. By induction hypothesis, there is a collection of cliques $\mathcal{K}''$ of $G''$ whose union covers $V(G'')$ such that $\sum_{Q \in \mathcal{K''}} |Q| = |V(G'')| - |V(G)| + \sum_{Q \in \mathcal{K}} |Q|$.

    Note that $G'$ is an inflation of $G''$ obtained from replacing a vertex $w$ with a clique of size $2$. Since the union of $\mathcal{K}''$ covers $V(G'')$, there is a clique $K'' \in \mathcal{K}''$ such that $w \in K''$. Let $v \in V(G') \setminus V(G'')$, and note that $v$ and $w$ are adjacent-twins. Let $K':= K'' \cup \{ v \}$. $K'$ is a clique in $G'$ since $v$ and $w$ are adjacent-twins, so $v$ is adjacent to each vertex in $K''$ in $G'$. Then $\mathcal{K}':= \mathcal{K}'' \cup \{ K' \} \setminus \{ K'' \} $ is a collection of cliques of $G'$ satisfying the desired properties.
\end{proof}
\begin{thm}
    \label{thm:fcc}
    If a graph $G$ with $\alpha(G) = 2$ has four cliques $(Q_1, Q_2, Q_3, Q_4)$ whose union covers $V(G)$ with $\sum_{i=1}^4 |Q_i| \geq |V(G)| + 2$, then any proper inflation of $G$ satisfies \nameref{conj:hc_half}.
\end{thm}
\begin{proof}
    Let $G'$ be a proper inflation of $G$. By \cref{lem:cc2}, there are cliques of $G'$ $(Q_1', Q_2', Q_3', Q_4')$ whose union covers $V(G')$ such that
    \begin{equation*}
        \sum_{i=1}^4 |Q_i'| = |V(G')| - |V(G)| + \sum_{i=1}^4 |Q_i| \geq |V(G')| + 2.
    \end{equation*}
    Therefore,
    \begin{equation*}
        4\omega(G') \geq \sum_{i=1}^4 |Q_i'| \geq |V(G')| + 2,
    \end{equation*}
    implying
    $\omega(G') \geq \left\lceil \frac{|V(G')| + 2}{4}\right\rceil$, and \cref{thm:seagullsN/4} gives the result.
\end{proof}

\subsection{Connected Matchings and Weakenings of Hadwiger's Conjecture}
\label{ss:connected_matchings}
This section summarises results for $\mathfrak{C} \in \{\text{\nameref{conj:hc_epsilon}, \nameref{conj:linear_cm}, \nameref{conj:4-CM}}\}$, which are weakenings of \nameref{conj:hc_chi}.
Let $\mathdefn{\cm(G)}$ be the largest size of a connected matching in $G$, and recall that $\mathdefn{\had_2(G)}$ is the order of the largest clique model in $G$ such that each branch set has one or two vertices. It is straightforward to see that $\cm(G) \leq \had_2(G)$. \citet{furedi2005connected} proved that \nameref{conj:4-CM} holds for graphs with small $|V(G)|$.

\begin{thm}[\citet{furedi2005connected}]
    \label{thm:t=16}
    For each $t \in \{1, \dots, 17\}$, every graph $G$ with $\alpha(G) = 2$ and $|V(G)| \geq 4t - 1$ has a connected matching of size $t$. Therefore, every graph $G$ with $\alpha(G) = 2$ and $|V(G)| \leq 67$ satisfies \nameref{conj:4-CM}.
\end{thm}
The bound was later improved to $t \in \{1, \dots, 22\}$ by \citet{chen2025connected}. Therefore, every graph $G$ with $\alpha(G) = 2$ and $|V(G)| \leq 87$ satisfies \nameref{conj:4-CM}. See \citet{chen2025connected} for properties of a counterexample to \nameref{conj:4-CM}.

An unexpected connection between \nameref{conj:hc_half} and \nameref{conj:4-CM} was discovered by \citet{cambie2021hadwiger}. He proved that \nameref{conj:hc_half} implies \nameref{conj:4-CM}. We present the proof here, since it is short and illuminating.
\begin{lem}[\citet{cambie2021hadwiger}]
    \label{lem:cambie1}
    For any integer $t \geq 1$, if $G$ is a graph with $\alpha(G) = 2$, $|V(G)| = 4t - 1$ and $\cm(G) \leq t - 1$, then $\omega(G) \leq \cm(G)$.
\end{lem}
\begin{proof}
    If $G$ is disconnected, then $G$ is the union of two complete graphs. One component of $G$ has order at least $2t$, implying $\cm(G) \geq t$.
    Hence $G$ is connected.
    Note that $\cm(G) \geq \floor{\omega(G) /2}$, so $\omega(G) \leq 2t-1$. By \cref{obs:non-neighbours}, $\delta(G) \geq |V(G)|-1 - \omega(G)\geq  2t-1$.
    Let $A$ be a clique of order $\omega(G)$ and let $B=V(G) \setminus A.$
    Let $M$ be the largest matching in the bipartite subgraph induced by the edges between $A$ and $B$. Any matching from $A$ to $B$ is connected since $A$ is a clique, so $\left|M\right| \leq \cm(G) \leq t-1$.
    Suppose for contradiction that $\omega(G) > \cm(G) \geq |M|$.
    This implies there is a vertex $x \in A$ unmatched by $M$. However, there is no neighbour of $x$ contained in $B - V(M)$, otherwise this contradicts the maximality of $M$. Therefore, all neighbours of $x$ are contained in $A \cup V(M)$. Since $\deg_G(x) \geq 2t - 1$, $|A \cup V(M)| \geq 2t$. Observe that
    $2|M| + |A \setminus V(M)| = |V(M) \cup (A \setminus V(M))| = |A \cup V(M)| = |A| + |M| \geq 2t$.
    Therefore, $|A \setminus V(M)| \geq 2t - 2|M|$, implying there are at least $2t - 2|M|$ vertices in $A$ left unmatched by $M$. We can extend the matching $M$ by pairing the remaining unmatched vertices of $A$, obtaining a connected matching of size $t$, a contradiction.
\end{proof}

\begin{lem}[\citet{cambie2021hadwiger}]
    \label{lem:cambie2}
    For every integer $t \geq 1$, if $G$ is a graph with $\alpha(G) = 2$, $|V(G)| = 4t - 1$, and $\cm(G) \leq t - 1$, then $\had_2(G) \leq \cm(G)$.
\end{lem}
\begin{proof}
    Choose a $K_{\had_2(G)}$-model $\mathcal{M}$ of $G$ with $k_1$ singletons and $k_2$ edges so that $k_2$ is maximised. Thus, $k_1 + k_2 = \had_2(G)$.
    Since $\cm(G) \leq t-1$, $k_2 \leq t- 1$.
    Suppose to the contrary that $\left|\mathcal{M}\right| = k_1 + k_2 > \cm(G)$. Then $k_1 \geq 1$, and let $x$ be one of the $k_1$ singletons in $\mathcal{M}$.
    \cref{lem:cambie1} asserts that $\omega(G) \leq \cm(G) \leq t -1$, so by \cref{obs:non-neighbours} $ \delta(G) \geq |V(G)|-1 - \omega(G)\geq  3t-1$. However, the collection $\mathcal{M}$ contains at most $k_1+2k_2\leq 3(t-1) = 3t - 3$ vertices, implying that $x$ has a neighbour outside $\mathcal{M}$. Pairing $x$ with this neighbour contradicts the maximality of $k_2$.
\end{proof}
\begin{thm}[\citet{cambie2021hadwiger}]
    \label{thm:cambie}
    \nameref{conj:hc_half} implies \nameref{conj:4-CM}.
\end{thm}
\begin{proof}
    We show the contrapositive. Suppose \nameref{conj:4-CM} is false and \nameref{conj:hc_half} holds. Let $G$ be a graph with $4t - 1$ vertices with $\cm(G) \leq t- 1$. Since \nameref{conj:hc_half} holds, there is a $K_{2t}$-model of $G$. For each $i \geq 1$, let $k_i$ be the number of branch sets of size $i$, and let $k:= \sum_{i \geq 3} k_i$. \cref{lem:cambie2} implies $k_1 + k_2 \leq \had_2(G) \leq t-1$. However, $k_1 + k_2 + k = 2t$ and $k_1 + 2k_2 + 3k \leq 4t - 1$. Thus, $
        6t = 3(k_1+k_2+k) \leq k_1+2k_2+3k + 2(k_1+k_2) \leq 4t-1 +2(t-1)=6t-3$,
    a contradiction.
\end{proof}
We end this section with asymptotics. Let $\mathdefn{f}: \N \to \N$ be a function such that every graph $G$ with $\alpha(G) = 2$ and $|V(G)| \geq t$ has a connected matching of size at least $f(t)$. \nameref{conj:hc_epsilon}, \nameref{conj:linear_cm} and \nameref{conj:4-CM} assert $f \in \Omega(t)$. The asymptotic magnitude of the Ramsey number $R(3,k)$ was determined by \citet{Kim95}, who proved that a $t$-vertex graph $G$ with $\alpha(G) = 2$ has $\omega(G) \in \Omega(\sqrt{t \log t})$. Since the complete graph on $2t$ vertices contains a connected matching of size $t$, $f(t) \in \Omega(\sqrt{t \log t})$. There have been several subsequent improvements. \citet{furedi2005connected} showed $f(t) \in \Omega(t^{2/3})$. Later, \citet{Blasiak07} showed $f(t) \in \Omega(t^{4/5})$. \citet{Fox10} currently holds the record, proving $f(t) \in \Omega(t^{4/5} \log^{1/5}t)$.

\subsection{Other Results}
Instead of asking for complete minors of order $\ceil{|V(G)|/ 2}$, what if we allow some missing edges? \citet{NS26a} proved that every graph $G$ with $\alpha(G) = 2$ contains a minor $H$ with $|V(H)| = \ceil{|V(G)|/ 2}$ and  $ |E(H)| \geq (\gamma - o(1)) \binom{|V(H)|}{2}$, where $\gamma \approx 0.986882$. Using ideas from \citep{NS26a}, \citet{yip2025dense} proved a dense variant of \nameref{conj:linear_cm}: for any $\varepsilon > 0$, there exists $c := c(\varepsilon)$ such that for every integer $t \geq 1$, every graph $G$ with $\alpha(G) \leq 2$ and $|V(G)| \geq ct$ contains a matching $M$ with $|M| = t$ such that the density of adjacent edge pairs is at least $1 - \varepsilon$.

A graph $G$ is \defn{claw-free} if it is $K_{1,3}$-free (that is, no vertex has three pairwise non-adjacent neighbours). Although every graph with $\alpha(G) = 2$ is claw-free, \nameref{conj:cdm} is false for the larger class of claw-free graphs (e.g., $C_7$). For claw-free graphs, \citet{chudnovsky_approximate_2010} proved a linear version of Hadwiger's Conjecture: every claw-free graph $G$ with no $K_{t+1}$-minor is $\floor{3t/2}$-colourable. For claw-free graphs, \citet{FradkinClique12} proved that the only exceptions to \cref{conj:indep} are graphs with $\alpha(G) = 2$.
\begin{thm}[\citet{FradkinClique12}]
    Every connected claw-free graph $G$ with $\alpha(G) \geq 3$ satisfies $\had(G) \geq |V(G)|/\alpha(G)$.
\end{thm}
A graph $G$ is \defn{quasi-line} if the neighbourhood of each vertex is the union of two cliques. Line graphs are quasi-line, and all quasi-line graphs are claw-free. However, there are graphs with $\alpha(G) = 2$ (e.g., the $5$-wheel) that are not quasi-line. \citet{reed_hadwigers_2004} first proved Hadwiger’s Conjecture for line graphs of multigraphs. This was later extended to quasi-line graphs by \citet{CO08}, who proved $\had(G) \geq \chi(G)$ for any quasi-line graph $G$. Observe that any inflation of a quasi-line graph is a quasi-line graph.

Hadwiger's Conjecture has been verified for several special classes of graphs with $\alpha(G) = 2$. \citet{PST03} proved \nameref{conj:shc_half} for inflations of small graphs: any inflation of $G$ with $|V(G)| \leq 11$ and $\alpha(G) = 2$ satisfies \nameref{conj:shc_half}.
For $d \geq 1$, the \defn{Andr\'asfai graph with parameter $d$}, denoted \defn{$\Gamma_d$}, is the $d$-regular triangle-free diameter-$2$ graph with $V(\Gamma_d) = \{v_1, \dots, v_{3d-1}\}$, and two vertices $v_i$ and $v_j$ are adjacent if and only if $|i - j| \equiv 1 \pmod{3}$. \citet{PST03} proved \nameref{conj:cdm}\footnotemark\footnotetext{They actually showed that any inflation of $G$ satisfies \nameref{conj:shc_half}. However, their proof can be modified to establish this.} for inflations of Andr\'asfai graphs, which includes the $5$-cycle.
Note that $\Gamma_2 \cong C_5$.
Lastly, \citet{PPT16} showed that for each $p \geq 1$, if $G$ is an inflation of $\overline{C_{2p+1}}$, then $G$ satisfies \nameref{conj:hc_chi}. We generalise this result in \cref{cor:anticycles}.

\section{Results for Special Graphs}
\label{section:results_special}
As \citet{Blasiak07} noted, it is difficult to identify possible counterexamples to Hadwiger's Conjecture for the $\alpha(G) = 2$ case. The complement of a minimal counterexample is triangle-free with diameter $2$ (\cref{thm:minimal_properties}), and the structure of these graphs is surprisingly rich. However, if we impose some additional conditions, then the diversity of examples quickly drops.
This section establishes Hadwiger's Conjecture for several interesting classes of graphs with $\alpha(G) = 2$.

\subsection{\boldmath Complements of Graphs with Girth At Least \texorpdfstring{$5$}{5}}
\label{ss:girth5}
Graphs with girth $5$ are triangle-free, and \citet{HS60} famously proved that there are finitely many graphs with diameter $2$ and girth $5$. In fact, the only known diameter-$2$ girth-$5$ graphs are the $5$-cycle, the Petersen graph, and the Hoffman-Singleton graph.
Any other diameter-$2$ girth-$5$ graph is $57$-regular with $3250$ vertices, and it is open if such a graph exists.
It is therefore natural to verify Hadwiger's Conjecture for the complements of these graphs and their inflations. We prove that \nameref{conj:cdm} holds for inflations of these graphs. Moreover, we establish a stronger result without the diameter $2$ condition and allowing larger girths:

\begin{thm}
    \label{thm:girth_5}
    Any inflation of the complement of a graph with girth at least $5$ satisfies \nameref{conj:cdm}.
\end{thm}
Graphs with girth at least $5$ are $C_4$-free, so their complements are $\overline{C_4}$-free, and therefore they satisfy \nameref{conj:cdm}, since every graph $H$ with $|V(H)| \leq 4$ and $\alpha(H) \leq 2$ is \nameref{conj:cdm}-unavoidable (see \cref{ss:induced}). Therefore, the main claim of \cref{thm:girth_5} is that \nameref{conj:cdm} continues to hold for their inflations, which may contain $\overline{C_4}$ as a subgraph.
\begin{proof}
    We may assume that $G$ is a connected inflation of a graph $H$, where $\overline{H}$ has girth at least $5$. By definition, $G$ is obtained by replacing each vertex $x \in V(H)$ with a clique $C_x$. Let $H'$ be the subgraph of $H$ induced by the set $\{x \in V(H): |C_x| > 0\}$. Then $G$ is a proper inflation of $H'$. Since $G$ is connected, $H'$ is connected. If $H'$ is $C_5$-free, then because $\overline{C_5} \cong C_5$ and $C_5$ is adjacent-twin-free, by \cref{lem:blow_up_avoidant}, $G$ is $C_5$-free as well. Therefore, by \cref{thm:c5}, $G$ satisfies \nameref{conj:cdm}. Therefore, we may assume $H'$ contains an induced $C_5$. Label its vertices $a,b,c,d,e$ clockwise.

    By \cref{obs:preserve}, $\alpha(H') = 2$, and thus the neighbourhood of each vertex $x \in V(H') \setminus \{a,b,c,d,e\}$ contains three consecutive vertices of $C_5$. Suppose for contradiction that $|N_{H'}(x) \cap \{a,b,c,d,e\}| = 3$ for some $x \in V(H') \setminus \{a,b,c,d,e\}$. Without loss of generality, $N_{H'}(x) \cap \{a,b,c,d,e\} = \{a,b,c\}$. Then $de, bx \in E(H')$ are non-adjacent edges in $H'$, so $dxeb$ is a $4$-cycle in $\overline{H'}$, which contradicts that the girth of $\overline{H'}$ is at least $5$.
    Therefore, every vertex $x \in V(H') \setminus \{a,b,c,d,e\}$ is adjacent to at least four vertices in $\{a,b,c,d,e\}$.

    Relabelling the vertices of $C_5$ anticlockwise and mirroring the vertices if needed, we may assume that $|C_a| = \min\{|C_a|, \dots, |C_e|\}$ and $|C_c| \leq |C_d|$. Let $M_1$ be a matching from $C_a$ to $C_e$ saturating $C_a$, and let $M_2$ be a matching from $C_c$ to $C_d$ saturating $C_c$. Since $G$ is a proper inflation of $H'$, $|M_1| = |C_a| \geq 1$ and $|M_2| = |C_c| \geq 1$, so $M_1 \cup M_2$ is a non-empty matching of $G$. For an edge $uv \in M_1$, since each vertex $x \in V(H') \setminus \{a,b,c,d,e\}$ is adjacent to at least four vertices in $\{a,b,c,d,e\}$, only vertices in $C_c$ are non-adjacent to $uv$. However, $M_2$ saturates $C_c$. Similarly, for an edge $uv \in M_2$, only vertices in $C_a$ are non-adjacent to $uv$, but $M_1$ saturates $C_a$. Thus, $M_1 \cup M_2$ is a non-empty connected dominating matching of $G$, as desired.
\end{proof}
\begin{cor}
    \label{thm:girth_5_eg}
    Let $G$ be the complement of the $5$-cycle, or the complement of the Petersen graph, or the complement of the Hoffman-Singleton graph, or the complement of a hypothetical $57$-regular graph with girth $5$. Then any inflation of $G$ satisfies \nameref{conj:cdm}.
\end{cor}

\begin{cor}
    \label{thm:girth_5_hc}
    Any inflation of the complement of a graph with girth at least $5$ satisfies \nameref{conj:shc_chi}.
\end{cor}
\begin{proof}
    Let $G$ be an inflation of $H$ with $\alpha(G) = 2$, where $H$ is the complement of a graph with girth at least $5$. Suppose for contradiction that $G$ fails \nameref{conj:shc_chi}. Let $G'$ be a minimal induced subgraph of $G$ such that $G'$ fails \nameref{conj:shc_chi}. Then $G'$ is an inflation of $H$ and is a minimal counterexample to \nameref{conj:shc_chi}. By \cref{thm:equiv}, $G'$ fails \nameref{conj:shc_half}, implying that $G'$ is connected. By \cref{thm:girth_5}, $G'$ contains a non-empty connected dominating matching, a contradiction to Property \propref{prop:no_connected_dominating_matching} of \cref{thm:minimal_properties}.
\end{proof}

\citet{PPT16} showed that for each $p \geq 1$, if $G$ is an inflation of $\overline{C_{2p+1}}$, then $G$ satisfies \nameref{conj:hc_chi}. Their proof uses the Strong Perfect Graph Theorem \citep{CRST-AM06}. We strengthen this result without using the Strong Perfect Graph Theorem.
\begin{cor}
    \label{cor:anticycles}
    For each $k \geq 3$, if $G$ is an inflation of $\overline{C_k}$, then $G$ satisfies \nameref{conj:shc_chi}.
\end{cor}
\begin{proof}
    The case when $k = 3$ is straightforward.
    If $G$ is an inflation of $\overline{C_4}$, then $G$ is a disjoint union of two non-adjacent complete graphs, so $G$ satisfies \nameref{conj:shc_chi}. If $G$ is an inflation of $\overline{C_k}$ for $k \geq 5$, then by \cref{thm:girth_5_hc}, $G$ satisfies \nameref{conj:shc_chi}.
\end{proof}

\subsection{Complements of Triangle-Free Kneser Graphs}
\label{ss:kneser_graphs}

Complements of triangle-free Kneser graphs (see \cref{ss:goals_outline} for a definition) satisfy many of the properties of minimal counterexamples to Hadwiger's Conjecture in the $\alpha=2$ case (see \cref{tab:minimal_properties}). In particular, when $2k \leq n \leq 3k-1$, $K(n,k)$ is triangle-free, implying $\alpha(\overline{K(n,k)}) = 2$.
When $n = 3k - 1$, $\diam(K(n,k)) = 2$ (\citep{valencia2005diameter}, see Property \propref{prop:complement_diameter_2}). \citet{merino2023kneser} proved that for all values of $n$ and $k$, $K(n,k)$ is Hamiltonian. If $n$ and $k$ are chosen so that $|V(K(n,k))|$ is odd, then by taking alternate edges of a Hamiltonian cycle, for each vertex $v \in V(K(n,k))$, $K(n,k) - v$ has a perfect matching. By \cref{lem:chromatic_matching}, this implies that $\mu(K(n,k)) = (|V(K(n,k))| - 1) / 2$, so $|V(K(n,k))| = 2 \chi(\overline{K(n,k)}) - 1$ and $\overline{K(n,k)}$ is vertex-critical (see Properties \propref{prop:vertex_critical} and \propref{prop:size_equals_2chi_minus_1}).
For sufficiently large $k$, \citet{Thomassen05} proved that $\overline{K(3k-1,k)}$ is a counterexample to Haj{\'o}s’ Conjecture. Moreover, he wrote that `it does not seem obvious' that these graphs satisfy Hadwiger's Conjecture.
Motivated by this comment, \citet{XZ17} proved that Hadwiger's Conjecture holds for all complements of Kneser graphs, including when $\alpha(\overline{K(n,k)}) = 2$. This raises the question of whether Hadwiger's Conjecture holds for inflations of $\overline{K(n,k)}$. We show that inflations of $\overline{K(n,k)}$ for $2k \leq n \leq 3k - 1$ satisfy \nameref{conj:hc_chi}, strengthening the result of \citet{XZ17} in this special case.

\begin{lem}[\citep{SU97}]
    \label{lem:frac_chromatic_kneser}
    For integers $n,k \geq 1$ such that $2k \leq n \leq 3k - 1$, $\theta_f(\overline{K(n,k)}) = \chi_f(K(n,k))  = n/k$.
\end{lem}
\begin{lem}
    \label{lem:kneser_graphs}
    For integers $n,k \geq 1$ such that $2k \leq n \leq 3k - 1$, any inflation of $\overline{K(n,k)}$ with an even number of vertices satisfies \nameref{conj:shc_half}.
\end{lem}
\begin{proof}
    By \cref{lem:frac_chromatic_kneser}, $\theta_f(\overline{K(n,k)})  = n/k < 3$. Applying \cref{thm:bla} gives the result.
\end{proof}
\begin{lem}
    \label{lem:kneser_graphs2}
    For integers $n,k \geq 1$ such that $2k \leq n \leq 3k - 1$, any inflation of $\overline{K(n,k)}$ satisfies \nameref{conj:hc_chi}.
\end{lem}
\begin{proof}
    By \cref{lem:frac_chromatic_kneser}, $\theta_f(\overline{K(n,k)})  = n/k < 3$. Applying \cref{thm:124/33} gives the result.
\end{proof}
Note that when $2k \leq n \leq 3k - 2$, $\overline{K(n,k)}$ has a dominating edge. However, this does not imply \cref{lem:kneser_graphs,lem:kneser_graphs2} for $2k \leq n \leq 3k - 2$, since there are induced subgraphs of $\overline{K(n,k)}$ with no dominating edge.

The problem becomes harder if we deal with complements of generalised Kneser graphs. For integers $n,k,t \geq 1$, let \defn{$K(n,k,\geq t)$} be the graph with vertex set $\binom{[n]}{k}$, and two vertices are adjacent if and only if their corresponding subsets intersect in at least $t$ elements. Define $\mathdefn{K(n,k, \leq t)}:= \overline{K(n,k, \geq t+1)}$. So $K(n,k, \geq 1) \cong \overline{K(n,k)}$ and $K(n,k,\leq 0) \cong K(n,k)$. For integers $n, k, t \geq 1$ satisfying $2k - t \leq n < 3k - 3t + 3$, $\alpha(K(n, k, \geq t)) \leq 2$.
This is because if there are $k$-sets $S_1, S_2, S_3$ such that $|S_i \cap S_j| \leq t - 1$ for each distinct $i,j \in \{1, 2, 3\}$, then by the inclusion-exclusion principle,
\begin{equation*}
    |S_1 \cup S_2 \cup S_3| = |S_1| + |S_2| + |S_3| - |S_1 \cap S_2| - |S_1 \cap S_3| - |S_2 \cap S_3| + |S_1 \cap S_2 \cap S_3| \geq 3k - 3(t-1).
\end{equation*}
The graph $K(n, k, \leq t)$ is known as a \defn{generalised Kneser graph} in the literature.

A \defn{$t$-intersecting} family $\mathcal{F} \subseteq \binom{[n]}{k}$ is a collection of subsets such that if $A, B \in \mathcal{F}$, then $\left|A \cap B\right| \geq t$. Define, for each $r$ such that $t + 2r \leq n$,
\begin{equation*}
    \mathdefn{\mathcal{F}_{n,k,t,r}} := \left\{ A \in \binom{[n]}{k}: \left|A \cap \{ 1, \dots, t+2r \}\right| \geq t+ r \right\}.
\end{equation*}
$\mathcal{F}_{n,k,t,r}$ is a $t$-intersecting family of sets, since if $S = \{1, \dots, t + 2r\}$, then for any $A_1, A_2 \in \mathcal{F}_{n,k,t,r}$,
\begin{equation*}
    t + 2r \geq |(A_1 \cap S) \cup (A_2 \cap S)| = |A_1 \cap S| + |A_2 \cap S| - |A_1 \cap A_2 \cap S| \geq 2(t+r) - |A_1 \cap A_2 \cap S|.
\end{equation*}
Therefore, $|A_1 \cap A_2| \geq t$. For a fixed $r \geq 0$,
\begin{equation}
    \left|\mathcal{F}_{n,k,t,r}\right| = \sum_{i=t+r}^{t+2r} \binom{t+2r}{i} \binom{n - t - 2r}{k - i}. \tag{1} \label{eq:frac}
\end{equation}
\citet{ahlswede1997complete} proved that the maximum size of a $t$-intersecting family is achieved by one of the $\mathcal{F}_{n,k,t,r}$ families.
\begin{thm}[\citet{ahlswede1997complete}]
    \label{thm:akt}
    The maximum size of a $t$-intersecting family in $\binom{[n]}{k}$ is at most
    $\max\limits_{r: t + 2r \leq n} \left|\mathcal{F}_{n,k,t,r}\right|$.
\end{thm}
The case when $t = 1$ is a special case of the Erd\H{o}s-Ko-Rado theorem, where the maximum is achieved when $r = 0$. Define:
\begin{equation*}
    \mathdefn{f_{n,k,t}}:= \max_{r: t + 2r \leq n} \left|\mathcal{F}_{n,k,t,r}\right|.
\end{equation*}
By \cref{thm:akt}, $\omega(K(n,k,\geq t)) = f_{n,k,t}$.

\begin{thm}
    \label{thm:fractional_kneser_graphs}
    For integers $n,k,t \geq 1$, $
        \theta_f(K(n,k,\geq t)) = \frac{\binom{n}{k}}{f_{n,k,t}} $.
\end{thm}
\begin{proof}
    Let $r \geq 0$ be an integer such that $t + 2r \leq n$ and $\left|\mathcal{F}_{n,k,t,r}\right| = f_{n,k,t}$, and let $s := t + 2r$. Fix a subset $S$ of $[n]$ of size $s$. Define:
    \begin{equation*}
        \mathcal{F}_{n,k,t,r}^S := \left\{ A \in \binom{[n]}{k}: \left|A \cap S \right| \geq t+ r \right\}.
    \end{equation*}
    $\mathcal{F}_{n,k,t,r}^S $ is a $t$-intersecting family, and the corresponding vertices induce a complete graph in $K(n, k, \geq t)$.
    Let $\mathcal{K}:= \{ \mathcal{F}_{n,k,t,r}^S: S \in \binom{[n]}{s} \}$ be a collection of subsets of $\binom{[n]}{k}$.

    Let $v \in V(K(n,k,\geq t))$, and let $K \in \binom{[n]}{k}$ be the corresponding $k$-subset. The number of subsets in $\binom{[n]}{s}$ that intersect $K$ at exactly $i$ elements is $\binom{k}{i} \binom{n-k}{s-i}$. Therefore, there are
    $
        \sum_{i=t+r}^{t+2r} \binom{k}{i} \binom{n-k}{s-i}
    $
    subsets $S \in \binom{[n]}{s}$ such that $t + r \leq |K \cap S| \leq t + 2r$. For each such subset $S$, note that $K \in \mathcal{F}_{n,k,t,r}^S$. Therefore, each vertex $v \in V(K(n,k,\geq t))$ is in $\sum_{i=t+r}^{t+2r} \binom{k}{i} \binom{n-k}{s-i}$ cliques of $\mathcal{K}$.
    By \cref{lem:equiv_frac},
    \begin{align*}
        \theta_f(K(n,k,\geq t)) & \leq \frac{\binom{n}{s}}{\sum_{i=t+r}^{t+2r} \binom{k}{i} \binom{n-k}{s-i}}.
    \end{align*}
    Note that
    \begin{align*}
        \frac{\binom{n}{s}}{\binom{k}{i} \binom{n-k}{s-i}} & = \frac{i! (k-i)! (s-i)! (n-k-s+i)! n!}{k! (n-k)! s! (n-s)!} = \frac{\binom{n}{k}}{\binom{s}{i} \binom{n-s}{k-i}}.
    \end{align*}
    Consequently,
    \begin{align*}
        \theta_f(K(n,k,\geq t)) & \leq \frac{\binom{n}{k}}{\sum_{i=t+r}^{t+2r} \binom{s}{i} \binom{n-s}{k-i}} = \frac{\binom{n}{k}}{f_{n,k,t}},
    \end{align*}
    where we have used \cref{eq:frac} for the last equality. Since $\alpha(G) \chi_f(G) \geq |V(G)|$ for every graph $G$, by \cref{thm:akt}, $\chi_f(\overline{K(n,k,\geq t)}) \geq \binom{n}{k} / f_{n,k,t}$, as desired.
\end{proof}

\cref{thm:fractional_kneser_graphs} provides a sufficient criterion for $K(n, k, \geq t)$ to satisfy \nameref{conj:hc_chi}.

\begin{thm}
    \label{thm:generalised_kneser_graphs}
    Let $n, k, t \geq 1$ be integers with $2k - t \leq n < 3k - 3t + 3$. If
    $
        f_{n,k,t} \geq \left(\frac{1}{4}  + \varepsilon\right) \binom{n}{k},
    $
    where $\varepsilon:=1/62$, then any inflation of $K(n, k, \geq t)$ satisfies \nameref{conj:hc_chi}.
\end{thm}
\begin{proof}
    By \cref{thm:fractional_kneser_graphs}, $\theta_f(K(n,k,\geq t)) = \frac{\binom{n}{k}}{f_{n,k,t}} \leq 1/(\frac{1}{4} + \frac{1}{62}) = 124/33$. By \cref{thm:124/33}, any inflation of $K(n, k, \geq t)$ satisfies \nameref{conj:hc_chi}.
\end{proof}
The \defn{odd girth} of a graph $G$ is the length of the shortest odd cycle in $G$.
\begin{thm}[\citet{denley1997odd}]
    \label{thm:odd_girth_of_kneser_graph}
    For integers $n, k, t \geq 1$, the odd girth of the generalised Kneser graph $K(n,k, \leq t)$ is
    $
        2 \left\lceil\frac{k-t}{n - 2(k-t)} \right\rceil + 1.
    $
\end{thm}
\begin{thm}
    \label{thm:generalised_kneser_graphs2}
    For integers $n,k,t \geq 1$ with $2k - t \leq n < \frac{5}{2} (k-t)$, any inflation $G$ of $K(n, k,\geq t+1)$ satisfies \nameref{conj:cdm}.
\end{thm}
\begin{proof}
    Let $G$ be an inflation of $K(n, k,\geq t+1)$.
    Since $2k - t \leq n < \frac{5}{2} (k-t)$, by \cref{thm:odd_girth_of_kneser_graph}, the odd girth of $K(n,k, \leq t)$ is at least $7$. Therefore, $K(n, k,\leq t)$ is $C_5$-free. Since $C_5$ is adjacent-twin-free, and since $\overline{C_5} \cong C_5$, by \cref{lem:blow_up_avoidant}, $G$ is $C_5$-free. By \cref{thm:c5}, $C_5$ is \nameref{conj:cdm}-unavoidable. Therefore, $G$ satisfies \nameref{conj:cdm}.
\end{proof}

\subsection{Complements of Strongly Regular Triangle-Free Graphs}
\label{ss:srg}
For integers $\lambda \geq 0$ and $\mu \geq 1$, a $k$-regular graph $G$ on $n$ vertices is \defn{$(\lambda ,\mu)$-strongly regular} if every pair of adjacent vertices has $\lambda$ common neighbours, and every pair of distinct non-adjacent vertices has $\mu$ common neighbours. Our notation for such a graph is \defn{$\srg(n, k, \lambda, \mu)$}. A regular graph is \defn{strongly regular} if it is $(\lambda, \mu)$-strongly regular for some integers $\lambda \geq 0$ and $\mu \geq 1$.

The case $\lambda = 0$ corresponds to triangle-free strongly regular graphs. Complete bipartite graphs $K_{n,n}$ are \defn{trivial} examples. There are seven known non-trivial strongly regular triangle-free graphs: the $5$-cycle, Petersen graph, the Clebsch graph, the Hoffman-Singleton graph, the Gewirtz graph, the Mesner graph, and the Higman-Sims graph. It is open whether infinitely many non-trivial triangle-free strongly regular graphs exist \citep{devillers2024triangle}. \citet{Biggs71a} showed that for any $\mu \notin \{2, 4, 6\}$, only finitely many $(0, \mu)$-strongly regular graphs exist; see
\citep{Biggs09a,Biggs09,Biggs11,Elzinga03} for related results. If $\mu \neq 0$, then $G$ has diameter $2$. All the graphs in \cref{ss:girth5} are strongly regular, so it is natural to prove Hadwiger's Conjecture for the complements of triangle-free strongly regular graphs. \citet{xu_hadwigers_2025} proved Hadwiger's Conjecture for the complement of the Clebsch and Petersen graphs. Below, we extend this to inflations of complements of all known strongly regular triangle-free graphs, except the Higman-Sims graph.

\subsubsection{Complement of the Clebsch Graph}
\label{sss:clebsch}
For an integer $n \geq 1$, an \defn{$n$-dimensional hypercube $Q_n$} is the graph whose vertex set is $\{0,1\}^n$, with two vertices adjacent if and only if they differ in exactly one coordinate. Two vertices are \defn{opposite} in an $n$-dimensional hypercube if their corresponding vectors differ in all $n$ coordinates. The \defn{Clebsch graph} is obtained from a $4$-dimensional hypercube by adding edges between all pairs of opposite vertices. It is a strongly regular graph with parameters $\srg(16,5,0,2)$ \citep{Clebsch}.
\begin{lem}
    \label{lem:fraction_clebsch}
    If $G$ is the complement of the Clebsch graph, then $\theta_f(G) \leq 16/5$.
\end{lem}
\begin{proof}
    For each vertex $v \in V(G)$, let $S_v = V(G) \setminus (N_G(v) \cup \{v\})$. Since $G$ is $10$-regular, $|S_v| = 5$.
    Since $\alpha(G) = 2$, each $S_v$ is a clique in $G$. Each vertex $w \in V(G)$ is in exactly $5$ of $\{S_v : v \in V(G)\}$, corresponding to the $5$ non-neighbours of $v$ in $G$. By \cref{lem:equiv_frac}, $\theta_f(G) \leq 16/5$.
\end{proof}
\begin{lem}
    \label{lem:inflation_clebsch}
    Every inflation of the complement of the Clebsch graph satisfies \nameref{conj:hc_chi}.
\end{lem}
\begin{proof}
    Let $G$ be an inflation of the complement of the Clebsch graph.
    By \cref{cor:inflation_frac,lem:fraction_clebsch}, $\theta_f(G) \leq 16/5 < 124/33$. \cref{thm:124/33} gives the result.
\end{proof}

\subsubsection{Complement of the Mesner Graph}
\label{sss:mesner}
A \defn{Steiner system} with parameters $t, k, n \geq 1$, denoted $\mathdefn{S(t,k,n)}$, is a family of $k$-subsets of $[n]$, called \defn{blocks}, such that each $t$-element subset of $[n]$ is contained in exactly one block.
For example, the Fano plane is a Steiner system with parameters $S(2,3,7)$.
The total number of $t$-element subsets of $[n]$ is $\binom{n}{t}$, and each block contains $\binom{k}{t}$ of these subsets.
Since each $t$-subset is contained in exactly one block, the number of blocks is $\binom{n}{t}/\binom{k}{t}$.

The \defn{Mesner graph}, or the \defn{$M_{22}$} graph, is the unique graph generated by the Steiner system $S(3,6,22)$.
Its vertices correspond to the blocks of the Steiner system, and two vertices are adjacent if and only if their corresponding blocks are disjoint.
It is a strongly regular graph with parameters $\srg(77, 16, 0, 4)$ \citep{Mesner}.
Note that $\left|V(M_{22})\right| = 77$, which is the number of blocks of $S(3,6,22)$, and $M_{22}$ is a triangle-free graph with diameter $2$.
We verified computationally that $M_{22}$ contains every graph in \cref{tab:unavoidable_graphs} as an induced subgraph, so one cannot apply the techniques described in \cref{ss:induced}.

\begin{thm}
    \label{fractional_m22} $\theta_f(\overline{{M}_{22}}) \leq 22/6$.
\end{thm}
\begin{proof}
    For each $i \in \{1, \dots, 22\}$, let $\mathcal{S}_i \subseteq V(\overline{{M}_{22}})$ be the set of vertices corresponding to blocks containing $i$. Then $\mathcal{S}_i$ is a clique in $\overline{M_{22}}$. Each vertex of $\overline{M_{22}}$ is in exactly $6$ of these cliques. \cref{lem:equiv_frac} gives the result.
\end{proof}
\begin{lem}
    \label{lem:inflation_m22}
    Every inflation of $\overline{M_{22}}$ satisfies \nameref{conj:hc_chi}.
\end{lem}
\begin{proof}
    By \cref{fractional_m22}, $\theta_f(\overline{M}_{22}) \leq 22/6 < 124/33$. \cref{thm:124/33} gives the result.
\end{proof}

\subsubsection{Complement of the Gewirtz Graph}
\label{sss:gewirtz}
The Gewirtz graph is constructed as follows. Consider the Steiner system $S(3,6,22)$ defined in \cref{sss:mesner}. Choose an element of $\{1, \dots, 22\}$, and let the vertices of the Gewirtz graph be the blocks not containing this element. Two blocks are adjacent if and only if they are disjoint. The Gewirtz graph is strongly regular with parameters $\srg(56, 10, 0, 2)$ \citep{Gewirtz}.

By construction, the Gewirtz graph is an induced subgraph of the Mesner graph. Therefore, an inflation of the complement of the Gewirtz graph is an inflation of the complement of the Mesner graph. \cref{cor:gewirtz} follows directly from \cref{lem:inflation_m22}.

\begin{cor}
    \label{cor:gewirtz}
    Every inflation of the complement of the Gewirtz graph satisfies \nameref{conj:hc_chi}.
\end{cor}

\subsubsection{Complement of the Higman-Sims Graph}
\label{sss:higman_sims}
The Higman-Sims graph, introduced by \citet{higman1968simple}, is constructed by slightly modifying the Mesner graph. Take the Mesner graph described in \cref{sss:mesner} and add $22$ new vertices, each corresponding to an element of $\{1, \dots, 22\}$. Each new vertex is adjacent to a block in $M_{22}$ if its corresponding element is in the block. Finally, add a vertex adjacent to all $22$ new vertices. It is a strong regular graph with parameters $\srg(100, 22, 0, 6)$ \citep{HigmanSims}. Its complement has chromatic number $50$ and clique number $22$. Due to the small size of the clique number relative to the order of the graph, \cref{thm:icr1/4} or \cref{thm:fcc} cannot be applied to prove Hadwiger's Conjecture for inflations of the complement of the Higman-Sims graph. Moreover, we computationally checked that the Higman-Sims graph contains every graph in \cref{tab:unavoidable_graphs} as an induced subgraph.

Nonetheless, we found a $K_{50}$-model with branch sets of size $2$. Using the labelling of the vertices provided by the \texttt{SageMath} command,  {\texttt{graphs.HigmanSimsGraph().complement()}}, the branch sets are:
\begin{center}
    \vspace{-0.5em}
    \small{
        \texttt{
            [(55, 57), (81, 82), (90, 91), (50, 52), (59, 61), (51, 53), (26, 27), (44, 45), (60, 62), (86, 87), (0, 2), (40, 41), (10, 12), (1, 3), (36, 37), (83, 85), (64, 65), (6, 8), (32, 33), (96, 97), (4, 5), (20, 22), (46, 47), (70, 72), (14, 15), (78, 79), (66, 68), (74, 75), (11, 13), (28, 29), (92, 93), (7, 9), (24, 25), (71, 73), (88, 89), (16, 18), (42, 43), (58, 63), (98, 99), (21, 23), (67, 69), (38, 39), (48, 49), (80, 84), (17, 19), (34, 35), (94, 95), (76, 77), (54, 56), (30, 31)]
        }
    }
\end{center}

\begin{thm}
    \label{thm:higman_sims}
    The complement of the Higman-Sims graph satisfies \nameref{conj:shc_half}.
\end{thm}

Interestingly, we cannot prove Hadwiger's Conjecture for inflations of the complement of the Higman-Sims graph, despite it being only a slight modification of the Mesner graph.

\subsection{Complements of Eberhard Graphs}
\label{ss:eberhard}
We started our investigation by considering graphs with diameter $2$ and girth $5$, and there are finitely many such graphs \citep{HS60}. In fact, there are no graphs with diameter $2$ and girth $5$ with strictly more than $3250$ vertices.
Put another way (since every tree with diameter $2$ is a star and $K_{2,2} \cong C_4$), the only diameter $2$ triangle-free graph not containing $K_{2,2}$ as a subgraph with strictly more than $3250$ vertices is the star graph $K_{1, n-1}$.
This observation led Wood~\citep{devillers2024triangle} to make the following conjecture:

\begin{conj}[\citep{devillers2024triangle}]
    \label{conj:wood}
    For every integer $t \geq 2$, there exists an integer $n_0$ such that if $G$ is a triangle-free diameter $2$ graph that does not contain $K_{2,t}$ as a subgraph with $n \geq n_0$ vertices, then $G$ is the star graph $K_{1, n-1}$.
\end{conj}

In essence, \cref{conj:wood} asks if the number of interesting diameter $2$ triangle-free graphs not containing $K_{2,t}$ as a subgraph is finite. The condition that $G$ does not contain $K_{2,t}$ as a subgraph is a relaxation of the girth $5$ condition. For an integer $t \geq 2$, let $\mathdefn{\mathcal{W}_t}$ denote the class of graphs other than stars that have diameter $2$ and contain neither a triangle nor a $K_{2,t}$-subgraph.
If \cref{conj:wood} were true, then for any fixed $t$, one could enumerate all graphs in $\mathcal{W}_t$, take their complements, and prove \nameref{conj:hc_chi} for those graphs.

\cref{conj:wood} was disproved by \citet{ETT25}. They showed that $\mathcal{W}_t$ is infinite for $t \in \{3,5,7\}$. Their examples for $\mathcal{W}_3$ and $\mathcal{W}_5$ are based on so-called crooked graphs, while their examples for $\mathcal{W}_7$ are Cayley graphs, which we now describe.

Let $(\Gamma, +)$ be an abelian group, and suppose $S \subseteq \Gamma$ is inverse-closed\footnote{$S \subseteq \Gamma$ is \defn{inverse-closed} if for each $x \in S$, $-x \in S$. Some definitions of a Cayley graph require that $S$ generates $\Gamma$, but we do not require it here.}. The \defn{Cayley graph} of $\Gamma$ on $S$ is the graph $\mathdefn{\Cay(\Gamma, S)}$ with vertex set $\Gamma$, and two vertices $a,b \in \Gamma$ are adjacent if and only if $a - b \in S$. For a prime $p \equiv 11 \pmod{12}$, the \defn{Eberhard graph with parameter $p$} is the Cayley graph $\Cay(\F_p \times \F_p, S)$, where $S:= \{ (x, \pm x^2): x \in \F_p \setminus \{0\} \}$.
\begin{obs}
    \label{obs:cay}
    For each $a, b \in \F_p$, we have $(a,0) \notin S$, $(0,b) \notin S$. If $a \in \F_p \setminus \{0\}$, then $|S \cap \{(a, x) : x \in \F_p\} | = 2$ and if $(a,b_1), (a,b_2) \in S$, then $b_1 \equiv \pm b_2 \pmod{p}$.
\end{obs}
The graph $\Cay(\F_p \times \F_p, S)$ has several interesting properties; see \citep{ETT25}. For example, it is triangle-free with diameter $2$, and does not contain $K_{2,7}$ as a subgraph. Therefore, $\Cay(\F_p \times \F_p, S)$ is connected and\footnote{$\Cay(\Gamma, S)$ is connected if and only if $S$ generates $\Gamma$.} $S$ generates $\F_p \times \F_p$. For the rest of this section, let $G:=\Cay(\F_p \times \F_p, S)$. The complement $\overline{G}$ has independence number $2$, but has small clique number relative to its order, making the application of \cref{thm:icr1/4,thm:fcc} unsuitable, and we computationally verified that $G$ contains every graph in \cref{tab:unavoidable_graphs} as an induced subgraph.

However, despite $\overline{G}$ having many of the properties of a minimal counterexample to \nameref{conj:shc_chi}, we now prove that $\overline{G}$ satisfies \nameref{conj:shc_chi}.

\begin{lem}
    \label{lem:chromatic_ebhard}
    $\chi(\overline{G}) = \frac{p^2+1}{2}$.
\end{lem}
\begin{proof}
    \citet{stelow2017hamiltonicity} proved that if $\Gamma$ is abelian and $S$ is a generating inverse-closed subset of $\Gamma$, then $\Cay(\Gamma, S)$ is Hamiltonian. Taking alternating edges of the Hamiltonian cycle gives a matching of size $\frac{p^2 - 1}{2} $ in $G$. Therefore, $\mu(G) = \frac{p^2 - 1}{2} $. By \cref{lem:chromatic_matching},  $\chi(\overline{G}) = \frac{p^2+1}{2}$.
\end{proof}

\begin{thm}
    \label{thm:eberhard}
    $\overline{G}$ satisfies \nameref{conj:shc_chi}.
\end{thm}
\begin{proof}
    Recall that $p$ is a prime with $p \equiv 11 \pmod{12}$. Note that this implies $p \equiv 3 \pmod{4}$, so $-1$ is not a quadratic residue modulo $p$. The $p \equiv 11 \pmod{12}$ condition implies that $p > 3$. This ensures that $\frac{p-1}{2} \not\equiv \pm 1 \pmod{p}$. We use this fact repeatedly. Partition $V(\overline{G})$ into the following connected sets of size at most $2$:
    \begin{itemize}
        \item Type 1 sets: these are singleton sets of the form $\{ (i, 0) \}$ for $i \in \F_p \setminus \{\frac{p-1}{2}, p-1\}$.
        \item A Type 2 set: a single $2$-element set: $\{ (\frac{p-1}{2} , 0), (p-1, 0) \}$.
        \item Type 3 sets: these are subsets $\{ (j, i), (j, p-1-i) \}$ for $j \in \F_p$ and $i \in \{ 1, \dots, \frac{p-1}{2} - 1 \}$.
        \item Type 4 sets: these are the subsets $\{ (j, \frac{p-1}{2} ), (j, p-1) \}$ for $j \in \F_p$.
    \end{itemize}

    \begin{lem}
        \label{lem:minorConnection}
        Suppose there is a vertex $(a,b) \in V(\overline{G})$ that is not adjacent to a Type $3$ set $\{ (j, i), (j, p-1-i) \}$ in $\overline{G}$ for some $j \in \F_p$ and $i \in \{ 1, \dots, \frac{p-1}{2} - 1 \}$. Then $b \equiv \frac{p-1}{2} \pmod{p}$.
    \end{lem}
    \begin{proof}
        Since $(a,b)(j,i) \in E(G)$ and $(a,b)(j, p-1-i) \in E(G)$, we have $(a-j, b-i), (a-j, b-p+1+i) \in S$. This implies $a-j \neq 0$, and by definition of $S$, we have $b-i \equiv \pm (b+i+1) \pmod{p}$. If the sign is positive, then $2i \equiv -1 \pmod{p}$, but this does not occur because $i \in \{ 1, \dots, \frac{p-1}{2} - 1 \}$. If the sign is negative, then $b-i \equiv -(b+i+1) \pmod{p}$. Thus $2b \equiv -1 \pmod{p}$, so $b \equiv \frac{p-1}{2} \pmod{p}$.
    \end{proof}
    We check that the sets are pairwise adjacent in $\overline{G}$.
    \begin{itemize}
        \item Two Type $1$ sets are adjacent: if $(a,0)$ is not adjacent to $(b,0)$ for distinct $a, b \in \F_p \setminus \{\frac{p-1}{2}, p -1\}$, then $(b-a, 0) \in S$, but this contradicts \cref{obs:cay}.
        \item Two Type $2$ sets are adjacent: vacuously true as there is only one Type $2$ set.
        \item Two Type $3$ sets are adjacent: take a vertex $(a,b)$ in a Type $3$ set. Since $b \not\equiv \frac{p-1}{2} \pmod{p}$, \cref{lem:minorConnection} implies $(a,b)$ is adjacent to all Type $3$ sets.
        \item Two Type $4$ sets are adjacent: straightforward by \cref{obs:cay}.
        \item A Type $1$ set and a Type $2$ set are adjacent: straightforward by \cref{obs:cay}.
        \item A Type $1$ set and a Type $3$ set are adjacent: by \cref{lem:minorConnection}, any vertex $(a,b)$ not adjacent to a Type $3$ set has $b \equiv \frac{p-1}{2} \pmod{p}$. Therefore $(a,b)$ is not in a Type $1$ set.
        \item A Type $1$ set and a Type $4$ set are adjacent: if there is $i \in \F_p \setminus \{\frac{p-1}{2}, p -1\}$ and $j \in \F_p$ such that $(i,0)$ is not adjacent to $(j, \frac{p-1}{2})$ and $(j, p-1)$, then $(j-i, \frac{p-1}{2}), (j-i, p-1) \in S$. By \cref{obs:cay}, $\frac{p-1}{2} \equiv \pm (p-1) \equiv \pm 1  \pmod{p}$. This is a contradiction since $p > 3$.
        \item A Type $2$ set and a Type $3$ set are adjacent: by \cref{lem:minorConnection}, any vertex $(a,b)$ not adjacent to a Type $3$ set has $b \equiv \frac{p-1}{2} \pmod{p}$, implying that $(a,b)$ is not in a Type $2$ set.
        \item A Type $2$ set and a Type $4$ set are adjacent: consider the vertex $(j, p-1)$ in a Type $4$ set for some $j \in \F_p$. If $(j, p-1)$ is not adjacent to $(\frac{p-1}{2} , 0)$ and $(p-1, 0)$ in $\overline{G}$, then $(j - \frac{p-1}{2} , p-1) \in S$ and $(j - (p-1), p-1) \in S$. As $p \equiv 3 \pmod{4}$, $-1$ is not a quadratic residue mod $p$, which implies that $(j - \frac{p-1}{2} , p-1)$ is of the form $(a, -a^2)$ for some $a \in \F_p$, and $(j - (p-1), p-1)$ is of the form $(b, -b^2)$ for some $b \in \F_p$. Hence, $(j- \frac{p-1}{2})^2 \equiv (j - (p-1))^2 \equiv -(p-1) \equiv 1 \pmod{p}$. Thus, $j - \frac{p-1}{2}$ and $j - (p-1)$ are roots of the polynomial $x^2 - 1 \equiv 0 \pmod{p}$. Therefore, $j - \frac{p-1}{2}, j - (p-1) \in \{1, -1\}$. Since $p > 3$, $j - \frac{p-1}{2} \not\equiv j - (p-1) \pmod{p}$.

              Consider two cases. If $j - \frac{p-1}{2} \equiv 1 \pmod{p}$ and $j - (p-1) \equiv -1 \pmod{p}$, then the latter equation implies that $j \equiv p - 2 \pmod{p}$. The former equation implies that $j \equiv 1 + \frac{p-1}{2} \pmod{p}$. This implies that $p - 2 \equiv 1 + \frac{p-1}{2} \pmod{p}$, and thus $5 \equiv 0 \pmod{p}$, a contradiction.

              If $j - \frac{p-1}{2} \equiv -1 \pmod{p}$ and $j - (p-1) \equiv 1 \pmod{p}$, then the latter equation implies that $j \equiv 0 \pmod{p}$. Setting $j \equiv 0$ in the first equation implies that $\frac{p-1}{2} \equiv 1 \pmod{p}$, which is a contradiction since $p > 3$.

        \item A Type $3$ set and a Type $4$ set are adjacent: by \cref{lem:minorConnection}, any vertex $(a,b)$ not adjacent to a Type $3$ set has $b \equiv \frac{p-1}{2} \pmod{p}$. Hence a Type $4$ set contains a vertex adjacent to the Type $3$ set.
    \end{itemize}
    The sets form a $K_{(p^2 + p - 2) / 2}$-model of $\overline{G}$, which contains a $K_{(p^2 + 1) / 2}$-model. By \cref{lem:chromatic_ebhard}, $\overline{G}$ satisfies \nameref{conj:shc_chi}.
\end{proof}

\section{Open Questions}
\label{s:op}
\label{ss:potential_counteg}

This section lists classes of graphs that might be considered potential counterexamples to one of the variants of Hadwiger's Conjecture, or one of the related conjectures.

As discussed in \cref{ss:kneser_graphs}, while we established \nameref{conj:hc_chi} for inflations of $\overline{K(n,k)}$ when $\alpha(\overline{K(n,k)}) = 2$, we did not prove:
\begin{quest}
    \label{q:kneser}
    For integers $n,k \geq 1$ such that $2k \leq n \leq 3k - 1$, does any inflation of $\overline{K(n,k)}$ satisfy \nameref{conj:shc_chi} or \nameref{conj:cdm}?
\end{quest}
The following connection provides further motivation for \cref{q:kneser}.
Given graphs $G$ and $H$, a \defn{homomorphism} from $G$ to $H$ is a function $h: V(G) \to V(H)$ such that $h(u) h(v) \in E(H)$ for each $uv \in E(G)$. Note that for each $x \in V(H)$, $h^{-1}(x)$ is empty or an independent set. When such a homomorphism exists, we say that $G$ is \defn{homomorphic} to $H$. It is well-known \citep{SU97} that there is a $(n,k)$-colouring of $G$ if and only if $G$ is homomorphic to $K(n,k)$.
A positive answer to \cref{q:kneser} implies that every graph $G$ with $\theta_f(G) < 3$ satisfies \nameref{conj:shc_chi} or \nameref{conj:cdm}. This would answer a question of \citet{Blasiak07}  and \citet{CS12}.

For generalised Kneser graphs $G$ (\cref{ss:kneser_graphs}), we proved \nameref{conj:hc_chi} only when the inflated clique ratio of $G$ is roughly greater than $1/4$. It would be interesting to prove \nameref{conj:hc_chi} for generalised Kneser graphs with small inflated clique ratios.
\begin{quest}
    \label{q:g_kneser}
    For integers $n,k,t \geq 1$ such that $2k - t \leq n < 3k -3t + 3$, does $K(n,k,\geq t)$ satisfy \nameref{conj:hc_chi}?
\end{quest}
For complements of the Clebsch graph, the Mesner graph, and the Gewirtz graph (\cref{ss:srg}), we proved \nameref{conj:hc_chi} but not \nameref{conj:shc_chi}.
For the complement of the Higman-Sims graph, we proved no result about its inflations.
\begin{quest}
    \label{q:srg}
    Let $G$ be the complement of the Clebsch graph, the complement of the Mesner graph, or the complement of the Gewirtz graph. Does any inflation of $G$ satisfy \nameref{conj:shc_chi}?
\end{quest}
\begin{quest}
    \label{q:other_vtx_transitive}
    Does any inflation of the complement of the Higman-Sims graph satisfy \nameref{conj:hc_chi}?
\end{quest}
Recall that $\mathcal{W}_t$ is the class of graphs other than stars that have diameter $2$ and contain neither a triangle nor a $K_{2,t}$-subgraph.
\citet{ETT25} proved that $\mathcal{W}_t$ is infinite for $t \in \{3,5,7\}$. We proved that inflations of graphs in $\mathcal{W}_2$ satisfy \nameref{conj:cdm} (\cref{thm:girth_5}), and proved that complements of Eberhard graphs satisfy \nameref{conj:shc_chi} (\cref{thm:eberhard}).
\begin{quest}
    \label{q:eberhard}
    For an integer $t \geq 3$, let $G$ be the complement of a graph in $\mathcal{W}_{t}$. Does any inflation of $G$ satisfy \nameref{conj:hc_chi}?
\end{quest}

Cayley graphs provide a way to construct triangle-free vertex-transitive graphs of diameter $2$. If $(\Gamma, +)$ is an abelian group and $S$ generates $\Gamma$, then $\Cay(\Gamma, S)$ is connected. If $0 \notin S$, then $\Cay(\Gamma, S)$ does not have self-loops, and $\overline{\Cay(\Gamma, S)} \cong \Cay(\Gamma, \Gamma \setminus (S \cup \{0\}))$. A subset $S \subseteq \Gamma$ with $0 \notin S$ is \defn{sum-free} if there does not exist (possibly non-distinct) $x, y, z \in S$ such that $x+y=z$. $S \subseteq \Gamma$ is \defn{sum-free-maximal} if for any non-zero $z \in \Gamma \setminus S$, $S \cup \{z\}$ is not sum-free.
\begin{lem}[Lemma 5, \citep{devillers2024triangle}]
    Let $(\Gamma, +)$ be an abelian group and $S$ be an inverse-closed subset such that $0 \notin S$. The following are equivalent for $G:= \Cay(\Gamma, S)$.
    \begin{itemize}
        \item $G$ is triangle-free if and only if $S$ is sum-free.
        \item $G$ has diameter $2$ if and only if $S$ is sum-free-maximal.
    \end{itemize}
\end{lem}
\begin{quest}
    \label{q:cayley}
    Let $(\Gamma, +)$ be an abelian group and $S$ be an inverse-closed generating subset of $\Gamma$ such that $0 \notin S$, and $S$ is sum-free-maximal. Does $\overline{\Cay(\Gamma, S)}$ satisfy \nameref{conj:hc_chi}? \nameref{conj:hc_chi} is open even for $\overline{\Cay(\F_p, S)}$, where $p$ is an odd prime and $S$ is an inverse-closed generating sum-free-maximal subset of $\F_p$.
\end{quest}

The following is known as the \defn{triangle-free process} \citep{Bohman09}. Let $n \geq 1$ be an integer. Start with $G_0$, an empty graph on $n$ vertices. For each $i \geq 1$, let $S_i:=\{uv \in E(\overline{G_{i-1}}): G_{i-1} + uv \text{ does not contain a triangle}\}$. Form $G_i$ by adding a randomly chosen edge $e_i \in S_i$ to $G_{i-1}$. The triangle-free process yields a random edge-maximal triangle-free graph on $n$ vertices. Its complement is an edge-minimal graph $G$ of independence number $2$. Edge-minimality ensures that $G$ has no dominating edge. \citet{Bohman09} proved these graphs provide lower bounds for the Ramsey number $R(3, t)$. In particular, the clique number of $G$ is asymptotically almost surely $\Theta(\sqrt{n \log n})$. As $n \to \infty$, $\omega(G)/|V(G)| \to 0$, which means that \cref{thm:icr1/4,thm:fcc}
cannot be applied. These seem to be promising places to look for counterexamples. It was also noted by \citet{Carter22} that all graphs of maximal order that are $K_8$-free with independence number $2$ contain every graph in \cref{tab:unavoidable_graphs} as an induced subgraph, which is surprising. (There are $477,142$ such graphs on $27$ vertices, which is the maximum order of a $K_8$-free graph with independence number $2$, since $R(3, 8) = 28$; see \citep{mckay_R38_1992}.)

\begin{quest}
    Let $G$ be the complement of a random graph generated by the triangle-free process with $|V(G)| \geq 4t - 1$. As $t \to \infty$, does $G$ contain a connected matching of size $t$ asymptotically almost surely\footnote{The \defn{Erd\H{o}s--R\'enyi random graph $G(n, p)$} is the $n$-vertex graph in which each edge is present independently with probability $p$. \citet{bollobas_hadwigers_1980} proved that for constant $p \in (0,1)$, asymptotically almost surely, $\had(G(n,p)) \geq \chi(G(n,p))$. However, this does not imply results for graphs with independence number $2$, because for constant $p \in (0,1)$, $\lim_{n \to \infty}\PP(\alpha(G(n,p)) = 2) = 0$ (Proposition 11.3.1, \citep{Diestel05}).}?
\end{quest}

\fontsize{10pt}{11pt}
\selectfont
\bibliographystyle{DavidNatbibStyle}
\bibliography{DavidBibliography}
\end{document}

%% file: implication_diagram.tex
\begin{tikzcd}
    & \text{\nameref{conj:cdm}} \arrow[ld] \arrow[rd, "\text{\cref{thm:equiv}}"]                 &                                                                                       \\
    \text{\nameref{conj:shc_chi}} \arrow[d]                                                                   &                                                              & \text{\nameref{conj:shc_half}} \arrow[ll, <->, "\text{\cref{thm:equiv}}"] \arrow[d] \\
    \text{\nameref{conj:hc_chi}} \arrow[rd] \arrow[rr, <->, "\text{\cref{thm:equiv}}"] &                                                              & \text{\nameref{conj:hc_half}}\arrow[ld, "\text{\citet{cambie2021hadwiger}}"]                                   \\
    & \text{\nameref{conj:4-CM}} \arrow[ld] \arrow[rd] &                                                                                       \\
    \text{\nameref{conj:linear_cm}}                                                             &                                                              & \text{\nameref{conj:hc_epsilon}} \arrow[ll, <->, "\text{\citet{KPT05}}"]
\end{tikzcd}

%% file: minimal_ppties_table.tex
\renewcommand{\arraystretch}{1.15} 
\begin{table}
    \centering
    \caption{Properties satisfied by minimal and minimum counterexamples to $\mathfrak{C}$, where $\mathfrak{C} \in \{\text{\nameref{conj:hc_chi}, \nameref{conj:shc_chi}}\}$. A \qmark \, indicates that the property is not known.}
    \label{tab:minimal_properties}
    \begin{tabular}{|p{8.1cm}|p{1.4cm}|p{1.4cm}|p{1.4cm}|p{1.4cm}|}
        \hline
        property                                                                                                                                                                                                                               &minimum c.e. to \nameref{conj:hc_chi} &minimal c.e. to \nameref{conj:hc_chi} &minimum c.e. to \nameref{conj:shc_chi} &minimal c.e. to \nameref{conj:shc_chi} \\
        \hline
        \propertylabel{prop:vertex_critical}{$G$ is $\chi(G)$-critical}                                                                                                                                                                        & \cmark                              & \cmark                              & \cmark                               & \cmark                               \\
        \hline

        \propertylabel{prop:not_decomposable}{$G$ is not decomposable}                                                                                                                                                                         & \cmark                              & \cmark                              & \cmark                               & \cmark                               \\
        \hline
        \propertylabel{prop:size_equals_2chi_minus_1}{$|V(G)| = 2 \chi(G) - 1$}                                                                                                                                                                & \cmark                              & \cmark                              & \cmark                               & \cmark                               \\
        \hline
        \propertylabel{prop:subgraph_vertex_critical}{$\forall xy \in E(\overline{G})$, $G-x-y$ is $(\chi(G) -1)$-critical}                                                                                                                           & \cmark                              & \cmark                              & \cmark                               & \cmark                               \\
        \hline
        \propertylabel{prop:complement_minus_x_pm}{$\forall v \in V(G), \overline{G} - v$ has a perfect matching}                                                                                                                              & \cmark                              & \cmark                              & \cmark                               & \cmark                               \\
        \hline

        \propertylabel{prop:no_connected_dominating_matching}{$G$ does not contain a non-empty connected dominating matching}                                                                                                                  & \cmark                              & \cmark                              & \cmark                               & \cmark                               \\
        \hline
        \propertylabel{prop:no_dominating_edges}{$\forall xy \in E(G), \alpha(G -xy) = 3$}                                                                                                  & \cmark                              & \cmark                              & \cmark                               & \cmark                               \\
        \hline
        \propertylabel{prop:kappa_geq_chi}{$\kappa(G) \geq \chi(G)$}                                                                                                                                                                           & \cmark                              & \cmark                              & \cmark                               & \cmark                               \\
        \hline
        \propertylabel{prop:delta_geq_chi}{$\delta(G) \geq \chi(G)$}                                                                                                                                                                           & \cmark                              & \cmark                              & \cmark                               & \cmark                               \\

        \hline
        \propertylabel{prop:hamiltonian}{$G$ is Hamiltonian}                                                                                                                                                                                   & \cmark                              & \cmark                              & \cmark                               & \cmark                               \\
        \hline
        \propertylabel{prop:factor_critical}{$\forall v \in V(G), G - v$ has a perfect matching}                                                                                                                                               & \cmark                              & \cmark                              & \cmark                               & \cmark                               \\
        \hline
        \propertylabel{prop:complement_diameter_2}{$\diam(\overline{G}) = 2$}                                                                                                                                                                  & \cmark                              & \cmark                              & \cmark                               & \cmark                               \\
        \hline
        \propertylabel{prop:B_nonempty}{$\forall xy \in E(\overline{G}), N_G(x) \cap N_G(y) \neq \varnothing$}                                                                                                                                        & \cmark                              & \cmark                              & \cmark                               & \cmark                               \\
        \hline
        \propertylabel{prop:B_non_neighbours}{$\forall xy \in E(\overline{G})$, for each $b \in N_G(x) \cap N_G(y)$, $b$ has a non-neighbour in $N_G(x) \setminus N_G(y)$ and a non-neighbour in $N_G(y) \setminus N_G(x)$}                           & \cmark                              & \cmark                              & \cmark                               & \cmark                               \\
        \hline
        \propertylabel{prop:A_C_common_neighbour}{$\forall xy \in E(\overline{G})$, $\forall a \in N_G(x) \setminus N_G(y)$ and $c \in N_G(y) \setminus N_G(x)$, $ac \in E(G)$ if and only if $a$ and $c$ have a common non-neighbour in $N_G(x) \cap N_G(y)$} & \cmark                              & \cmark                              & \cmark                               & \cmark                               \\
        \hline
        \propertylabel{prop:C5_condition}{Any two non-adjacent vertices are in an induced $C_5$}                                                                                                                                          & \cmark                              & \cmark                              & \cmark                               & \cmark                               \\
        \hline
            \propertylabel{prop:chi_at_least_7}{$\chi(G)  \geq 7$}                                                                                                                                                                                 & \cmark                              & \cmark                              & \qmark                               & \qmark                               \\
        \hline
        \propertylabel{prop:kappa_at_least_7}{$\kappa(G) \geq 7$}                                                                                                                                                                              & \cmark                              & \cmark                            & \qmark                               & \qmark                               \\
        \hline
        \propertylabel{prop:omega_at_most_chi_minus_3}{$\omega(G) \leq \chi(G) - 3$}                                                                                                                                                           & \cmark                              & \cmark                              & \qmark                               & \qmark                               \\
        \hline
        \propertylabel{prop:delta_geq_chi_plus_1}{$\delta(G) \geq \chi(G) + 1$}                                                                                                                                                                & \cmark                              & \cmark                              & \qmark                               & \qmark                               \\
        \hline
        \propertylabel{prop:ABC_sizes}{$\forall xy \in E(\overline{G})$, $2 \leq |N_G(x) \setminus N_G(y)|, |N_G(y) \setminus N_G(x)| \leq \chi(G) -4$, and $5 \leq |N_G(x) \cap N_G(y)| \leq 2\chi(G)-7$}                                                  & \cmark                              & \cmark                              & \qmark                               & \qmark                               \\
        \hline
        \propertylabel{prop:edge_critical}{$\forall xy \in E(G), \chi(G - xy) < \chi(G)$}                                                                                                                                                     & \cmark                              & \qmark                              & \qmark                               & \qmark                               \\

        \hline
        \propertylabel{prop:proper_minor_less_chi}{Every proper minor $H$ of $G$ has $\chi(H) < \chi(G)$}                                                                                                                                      & \cmark                              & \qmark                              & \qmark                               & \qmark                               \\
        \hline

    \end{tabular}
\end{table}

%% file: forbidden_induced.tex
\begin{table}[t]
  \centering
  \caption{The \textbf{complements} of all known induced maximal \nameref{conj:hc_chi}-unavoidable graphs. They are $\overline{H_1}, \dots, \overline{H_{33}}$, $\overline{B_7}$, $\overline{K_8}$, and $K_{1,6}$ in order (images from \citet{Carter22}).}
  \label{tab:unavoidable_graphs}
   \begin{tabular}{|c|c|c|c|c|c|}
    \hline
    \includegraphics[scale=0.85]{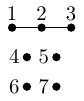}  & \includegraphics[scale=0.85]{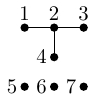} & \includegraphics[scale=0.85]{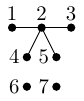}  & \includegraphics[scale=0.85]{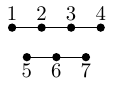} & \includegraphics[scale=0.85]{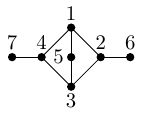} & \includegraphics[scale=0.85]{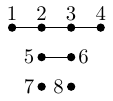} \\ \hline
    \includegraphics[scale=0.85]{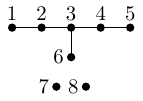}  & \includegraphics[scale=0.85]{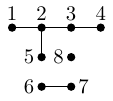} & \includegraphics[scale=0.85]{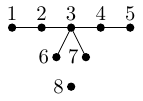}  & \includegraphics[scale=0.85]{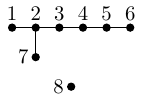} & \includegraphics[scale=0.85]{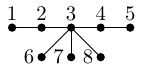} & \includegraphics[scale=0.85]{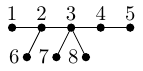} \\ \hline
    \includegraphics[scale=0.85]{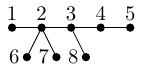} & \includegraphics[scale=0.85]{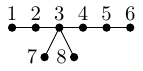} & \includegraphics[scale=0.85]{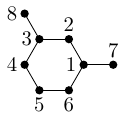} & \includegraphics[scale=0.85]{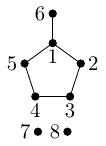} & \includegraphics[scale=0.85]{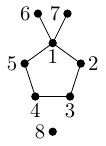} & \includegraphics[scale=0.85]{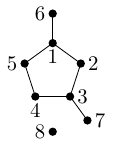} \\ \hline
    \includegraphics[scale=0.85]{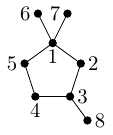} & \includegraphics[scale=0.85]{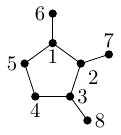} & \includegraphics[scale=0.85]{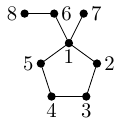} & \includegraphics[scale=0.85]{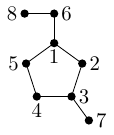} & \includegraphics[scale=0.85]{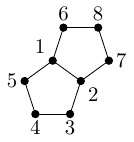} & \includegraphics[scale=0.85]{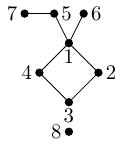} \\ \hline
    \includegraphics[scale=0.85]{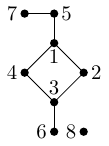} & \includegraphics[scale=0.85]{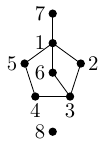} & \includegraphics[scale=0.85]{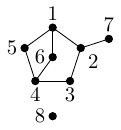} & \includegraphics[scale=0.85]{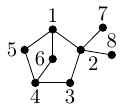} & \includegraphics[scale=0.85]{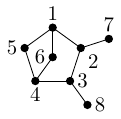} & \includegraphics[scale=0.85]{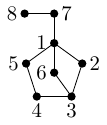} \\ \hline
    \includegraphics[scale=0.85]{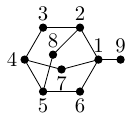} & \includegraphics[scale=0.85]{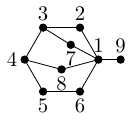} & \includegraphics[scale=0.85]{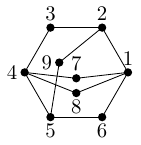} & \includegraphics[scale=0.85]{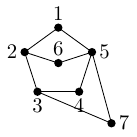} & \includegraphics[scale=0.85]{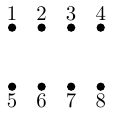} & \includegraphics[scale=0.85]{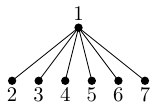} \\ \hline
  \end{tabular}
\end{table}